\documentclass[11pt,a4paper,reqno]{amsart}
\usepackage{amsthm,enumerate}
\usepackage{amsmath}
\usepackage{amssymb}
\usepackage{graphicx,color}
\usepackage[dvipsnames]{xcolor}
\usepackage{bm}

\usepackage[font=small]{caption}
\usepackage{subcaption}

\usepackage{multirow}
\usepackage[shortlabels]{enumitem}

\usepackage[normalem]{ulem}
\usepackage{cancel}

\usepackage[bookmarksopen=true,final]{hyperref}

\numberwithin{equation}{section}

\makeatletter
\theoremstyle{plain}
\newtheorem{thm}{Theorem}
  \theoremstyle{definition}
  
  \theoremstyle{remark}
  \newtheorem{rem}[thm]{Remark}
  \theoremstyle{plain}
  \newtheorem{prop}[thm]{Proposition}
  \theoremstyle{plain}
  \newtheorem{lem}[thm]{Lemma}
  \theoremstyle{plain}
  \newtheorem{cor}[thm]{Corollary}
 \theoremstyle{definition}
  \newtheorem{example}[thm]{Example}
  \theoremstyle{remark}
  \newtheorem*{rem*}{Remark}

  \theoremstyle{definition}

\usepackage{amsfonts}
\usepackage{mathrsfs}

\newtheorem*{op*}{Open problem}

\theoremstyle{plain}

\newcommand{\N}{\mathbb{N}}
\newcommand{\R}{{\mathbb{R}}}
\newcommand{\C}{{\mathbb{C}}}
\newcommand{\Z}{{\mathbb{Z}}}
\newcommand{\D}{{\mathbb{D}}}
\newcommand{\T}{{\partial\mathbb{D}}}
\newcommand{\dd}{{\rm d}}
\newcommand{\ii}{{\rm i}}
\newcommand{\ee}{{\rm e}}
\newcommand{\calZ}{\mathcal{Z}}
\newcommand{\calP}{\mathcal{P}}

\newcommand{\diag}{\mathop\mathrm{diag}\nolimits}

\newcommand{\Ran}{\mathop\mathrm{Ran}\nolimits}

\renewcommand{\Re}{\mathop\mathrm{Re}\nolimits}

\newcommand{\supp}{\mathop\mathrm{supp}\nolimits}
\newcommand{\Tr}{\mathop\mathrm{Tr}\nolimits}
\newcommand{\dist}{\mathop\mathrm{dist}\nolimits}

\DeclareMathOperator*{\wlim}{w-lim\ }

\makeatother

\begin{document}

\title[]{Ratio asymptotics and zero density for orthogonal polynomials with varying Verblunsky coefficients}

\author{Rostyslav Kozhan}
\address[Rostyslav Kozhan]{
	Department of Mathematics, Uppsala University, Box 480,
	75106 Uppsala, Sweden
	}
\email{kozhan@math.uu.se}

\author{Franti\v sek \v Stampach}
\address[Franti{\v s}ek {\v S}tampach]{
	Department of Mathematics, Faculty of Nuclear Sciences and Physical Engineering, Czech Technical University in Prague, Trojanova~13, 12000 Praha, Czech Republic
	}	
\email{stampfra@cvut.cz}



\date{\today}

\begin{abstract}
We study asymptotic behavior of orthogonal polynomials on the unit circle with varying Verblunsky coefficients $\alpha_{n,N}$  when the ratio $n/N$ converges as $n,N\to\infty$.
First, we give a streamlined proof of ratio asymptotics for orthogonal and paraorthogonal polynomials in the case of asymptotically constant and asymptotically periodic coefficients $\alpha_{n,N}$.	Second, we determine the asymptotic zero distribution of paraorthogonal polynomials in the locally constant and locally periodic regimes. Analogous results are obtained for orthogonal polynomials under a mild additional condition on the varying coefficients.
\end{abstract}

\maketitle

\section{Introduction}

We investigate orthogonal polynomials $\Phi_{n,N}$ on the unit circle determined by Verblunsky coefficients $\alpha_{n,N}$ that are varying with an additional index $N\in\N$. In other words, given a doubly-indexed sequence $\alpha_{n,N}$ from the unit disk $\D\equiv\{z\in\C \mid |z|<1\}$, the monic polynomial $\Phi_{n,N}$ of degree $n$ is generated by the Szeg{\H o} recurrence
\[
\Phi_{n+1,N}(z)=z\Phi_{n,N}(z)-\bar{\alpha}_{n,N}\Phi_{n,N}^{*}(z), \quad n\ge 0,
\]
with the initial condition $\Phi_{0,N}(z)=1$ for all $N\in\N$. Here and throughout, the bar denotes complex conjugation and $\Phi_{n,N}^{*}(z)=z^{n}\,\overline{\Phi_{n,N}\left(1/\bar{z}\right)}$ is the reversed polynomial.
Given an additional parameter $\beta$ in the unit circle $\T\equiv\{z\in\C \mid |z|=1\}$, the closely related concept of paraorthogonal polynomials $\Phi_{n,N}^{(\beta)}$ is, for each $N\in\N$, defined by
\[
	\Phi^{(\beta)}_{n+1,N}(z)=z\Phi_{n,N}(z)-\bar{\beta}\Phi_{n,N}^{*}(z), \quad n\ge 0.
\]

In this article, we study the asymptotic behavior of the polynomials $\Phi_{n,N}$, $\Phi_{n,N}^{(\beta)}$, and of their zeros, in the regime where both $n$ and $N$ tend to infinity in such a way that the ratio $n/N$ converges to a positive constant $t$; we will simply write $n/N\to t$.
For brevity, we will use the common abbreviations OPUC and POPUC for orthogonal polynomials on the unit circle
and paraorthogonal polynomials on the unit circle, respectively, throughout the paper.

\subsection{Informal summary of main results}

Let us first explain the nature of our main results without introducing all necessary notation and technical points.

In our first result (Theorem~\ref{thm:ratio}), we establish the ratio asymptotics for polynomials $\Phi_{n,N}$, $\Phi^*_{n,N}$, $\Phi_{n,N}^{(\beta)}$, as $n/N\to t$, assuming the varying Nevai--L\'{o}pez condition
\begin{equation}\label{eq:Verblunsky}
	\lim_{n/N\to t}\alpha_{n,N}=\alpha.
\end{equation}
While ratio asymptotic formulas are well known in the non-varying case ($\alpha_{n,N}\equiv \alpha_n$), see~\cite{BarLop,Khr02} and~\cite{simon_opuc1,simon_opuc2}, making the existence of such results expected, our approach via the operator theory of CMV matrices offers a streamlined proof that both extends these results to the varying case and provides a fresh insight into the classical setting.
In particular, it allows for an explicit identification of the limiting ratio in terms of standard objects of OPUC theory such as the Schur, Carathéodory, and $m$-functions, and it establishes uniform convergence on optimal domains, i.e., on arbitrary compact subsets of $\C$ which avoid the limiting zero set of the polynomials.

Another advantage of our approach is in its broader applicability. It allows to extend the first result, without any significant additional difficulty, to much wider generality, which we demonstrate on the asymptotically periodic setting when~\eqref{eq:Verblunsky} is replaced by the condition
\begin{equation}\label{eq:almostPeriodicIntro}
	\lim_{kp/N\to t}\alpha_{kp+j,N}=\alpha_{j}, \quad \mbox{for all }  j\in\{0,\dots,p-1\},
\end{equation}
where $p\in\N$ is a period. This constitutes our second result (Theorem~\ref{thm:ratioPeriodic}).
In the non-varying case, it improves upon a result of Barrios and L\'{o}pez~\cite{BarLop}, see Corollary~\ref{cor:BarLop}. 

Our next results concern the asymptotics of the zero-counting measures 
\begin{equation}
	\nu_{n,N}=\frac{1}{n}\sum_{j=1}^{n}\delta_{z_{j,n,N}},
\quad\mbox{ and }\quad
	\nu^{(\beta)}_{n,N}=\frac{1}{n}\sum_{j=1}^{n}\delta_{z^{(\beta)}_{j,n,N}},
\label{eq:count-meas_var}
\end{equation}
where $z_{j,n,N}$ and $z^{(\beta)}_{j,n,N}$ are zeros (counted with multiplicity) of $\Phi_{n,N}$ and $\Phi_{n,N}^{(\beta)}$, respectively, and $\delta_{z}$ is the Dirac unit mass measure supported at the one-point set~$\{z\}$.

Clearly, the mere conditions~\eqref{eq:Verblunsky} or~\eqref{eq:almostPeriodicIntro} are generally insufficient to guarantee the existence of weak limits for these measures, as they constrain only a small portion of the Verblunsky coefficients. This leads us to the so-called asymptotically locally constant setting, in which we assume that
\begin{equation}\label{eq:VerblunskyStrong}
	\lim_{n/N\to s}\alpha_{n,N}=\alpha(s), \quad \mbox{for a.e. } s\in [0,t],
\end{equation}
where $\alpha$ is an integrable function on the interval $[0,t]$. A model situation for this set-up is obtained by sampling a piece-wise continuous function $\alpha$ on a regular grid of $[0,t]$:
$\alpha_{n,N}=\alpha\left(n/N\right)$ for $n\leq tN$ (and, for instance, $\alpha_{n,N}=0$ for $n>tN$).

Our third result (Theorem~\ref{thm:popuc_asympt_distr}) states that, under condition~\eqref{eq:VerblunskyStrong}, the weak limit of the zero-counting measure $\nu^{(\beta)}_{n,N}$ of POPUC exists and equals an average measure:
\begin{equation}\label{eq:weak_limitIntro}
	\wlim_{n/N\to t} \nu_{n,N}^{(\beta)}=\frac{1}{t}\int_{0}^{t}\nu_{|\alpha(s)|}\dd s.
\end{equation}
Here the measure $\nu_{|\alpha(s)|}$ is the equilibrium measure of a circular arc determined by $|\alpha(s)|$.
If we additionally assume that $|\alpha_{n-1,N}|^{1/n}\to 1$ as $n/N\to t$, then the zero-counting measure $\nu_{n,N}$ of OPUC has the same weak limit as $n/N\to t$. This is our fourth result (Theorem~\ref{thm:opuc_asympt_distr}).

Finally, analogous results (Theorem~\ref{thm:asympt_distr_popuc_periodic} and Remark~\ref{rem:asymptPerOPUC}) hold in the asymptotically locally periodic setting, which assumes
\[
	\lim_{kp/N\to s} \alpha_{kp+j,N}=\alpha_{j}(s), \,\mbox{ for {a.e.} } s\in[0,t],
\]
where $\alpha_1,\ldots,\alpha_p$ are $p$ integrable functions on $[0,t]$. In this case, the measure $\nu_{|\alpha(s)|}$ from~\eqref{eq:weak_limitIntro} is replaced by the equilibrium measure of a certain finite-gap subset of $\T$.

\subsection{Relevant literature}

We benefit greatly from the comprehensive treatment of the general theory of OPUC presented in Simon's monographs~\cite{simon_opuc1, simon_opuc2}.

There are numerous works that address properties of the zeros of orthogonal and paraorthogonal polynomials on the unit circle, see, e.g., \cite{ABMV,AlfVig,BarLop,BlaSafSim,BreSee,Cas19,CasPet,DavSim,DenSim,ErdTur,golinski_maa99,Gol02,Khr03,KK,KilNen,KilSto,Siman16,Siman20,SimZero4,MFMS06a,MFMS06b,mha-saf_jat90,NevTot,Pak,SafTot,SimZero2,simon_opuc1,simon_opuc2,SimZero1,SimZero3,SimZeroRev,SimPOPUC,Sim09,SimTot,Sto,Won07}. The literature on asymptotics of orthogonal polynomials on the unit circle is so extensive that we do not attempt to survey it here. The results most relevant for our purposes, and on which we rely, are those in ~\cite{BarLop,golinski_maa99,Khr02,mha-saf_jat90,Sim09} and \cite[Sec. 8]{simon_opuc1}.

A parallel study formulated for orthogonal polynomials on the real line with varying Jacobi parameters has already received significant attention. Their asymptotic zero distribution was analyzed by Kuijlaars and Van Assche~\cite{kuijlaarsvanassche_jat99}. Related results had already appeared in the Szeg{\H o}-type limit formulas for Hermitian variable Toeplitz matrices due to Kac, Murdock and Szeg{\H o}~\cite{kms_jrma53}, later revisited by Tilli~\cite{tilli_laa98}; see also~\cite{bourget-etal_jjm18, pastur_mps96}. These works also influenced the development of locally Toeplitz sequences and their applications in numerical analysis~\cite{bar-ser_20,garoni-serracapizzano_vol1,garoni-serracapizzano_vol2}.

The article~\cite{kuijlaarsvanassche_jat99} had a significant impact on the orthogonal polynomial community, leading to several generalizations and extensions: to the asymptotically periodic case in~\cite{vanassche_ss99, fas-ser_03}, to discontinuous limiting functions in~\cite{kuijlaars-serracapizzano_jat01}, to Laurent orthogonal polynomials with varying recurrence coefficients in~\cite{coussement-vanassche_jpa05}, and to variable multiple orthogonal polynomials in~\cite{coussement-etal_tams08}. Our results cover the full range of~\cite{kuijlaarsvanassche_jat99,vanassche_ss99,kuijlaars-serracapizzano_jat01} in the unit-circle setting, including the analysis of both asymptotically periodic behavior and discontinuous sampling functions.

Finally, we note that the condition~\eqref{eq:VerblunskyStrong}
appears to be very natural from the perspective of discretizing Krein systems and one-dimensional Dirac operators~\cite{denisov_imrs06,Kre55}. The problem of understanding asymptotic zero densities in the varying setting is also closely related to questions in random matrix theory concerning global eigenvalue distribution of orthogonal polynomial ensembles on the unit circle, see, e.g.,~\cite{BreOfn25,DuiKoz}.

\subsection{Structure of the paper}

The paper is organized as follows.
Section~\ref{sec:prelim} collects notation and the necessary background from the OPUC theory and potential theory.

Section~\ref{sec:ratio_asym} derives ratio asymptotics for varying POPUC and OPUC. The asymptotically constant case is treated in Section~\ref{subsec:RatioConstant}, and the asymptotically periodic case in Section~\ref{subsec:RatioPeriodic}.

Section~\ref{sec:var_popuc} establishes the asymptotic zero distribution for POPUC with varying Verblunsky coefficients. We present two distinct proofs of the main result: one combining ratio asymptotics with a potential-theoretic argument in Section~\ref{sec:FirstProof}, and the other based on the method of moments in Section~\ref{sec:SecondProof}. The asymptotically periodic case is discussed in Section~\ref{subsec:distr_asympt_popuc_periodic}.

Section~\ref{sec:var_opuc} is devoted to the asymptotic zero distribution for OPUC with varying Verblunsky coefficients. A transition from POPUC to OPUC is handled via balayage in Section~\ref{subsec:balay}, and the main result is deduced in Section~\ref{subsec:distr_asympt_opuc}.

Finally, Section~\ref{sec:examples} presents several illustrative examples where the limiting densities are computed explicitly or expressed in terms of classical special functions, complete with numerical plots.

\section{Preliminaries}\label{sec:prelim}

This section reviews standard facts from the theory of orthogonal polynomials on the unit circle and potential theory that will be used later. The monographs \cite{simon_opuc1, simon_opuc2} serve as main references.

\subsection{Notation}

Throughout the paper we use the following notation: 
\begin{itemize}
\item $\N=\{1,2,3,\dots\}$ for positive integers;
\item $\N_0=\{0,1,2,\dots\}$ for non-negative integers;
\item $\D=\{z\in\C \mid |z|<1\}$ for the open unit disk;
\item $\overline{\D}=\{z\in\C \mid |z|\leq1\}$ for the closed unit disk;
\item $\T=\{z\in\C \mid |z|=1\}$ for the unit circle;
\item $\overline{\C}=\C\cup\{\infty\}$ for the Riemann sphere;
\item $\supp \mu$ for the support of a measure $\mu$;
\item $\dist(z,A)$ for the Euclidean distance of a point $z$ and a set $A$ in $\C$.
\item $\D^{\infty}=\{\{\alpha_{n}\}_{n=0}^{\infty} \mid \alpha_{n}\in\D \mbox{ for all } n\in\N_{0} \}$ for the set of sequences with ele\-ments from~$\D$;
\item $\ell^{2}(\N_{0})=\{\{x_{n}\}_{n=0}^{\infty} \mid \sum_{n=0}^{\infty}|x_{n}|^{2}<\infty\}$ for the Hilbert space of complex square summable sequences; 
\item $\langle\cdot,\cdot\rangle$ for the inner product on $\ell^{2}(\N_{0})$, linear in the second argument;
\item $e_{i}$ for the $i$-th vector of the standard basis of either $\C^{n}$ or $\ell^{2}(\N_{0})$ which is always clear from the context, i.e., $(e_{i})_{i}=1$ and $(e_{i})_{j}=0$ for $j\neq i$.
\end{itemize}
Other notation will be introduced by its first occurrence.


Concerning the doubly indexed sequences and limit relations, we adopt the notation used in~\cite{kuijlaarsvanassche_jat99}. This means that, for a doubly indexed sequence $X_{n,N}$ and $t>0$, we write
\[
\lim_{n/N\to t}X_{n,N}=X
\]
if and only if
\[
\lim_{j\to\infty}X_{n_j,N_j}=X,
\]
for any 
\begin{equation}
		 \{n_j\}_{j\in\N}, \{N_j\}_{j\in\N}\subseteq\N,
		\mbox{ such that }
		 n_j, N_j\to\infty \mbox{ and } n_j/N_j\to t, \mbox{ as } j\to\infty.
	\label{eq:subs}
\end{equation}
Whenever necessary, the precise meaning of the limit is also indicated. For example,
\[
\wlim_{n/N\to t}\mu_{n,N}={\mu}
\]
expresses the limit of the doubly indexed sequence of measures $\mu_{n,N}$ converging to $\mu$ weakly, i.e., in the weak$^{\star}$ topology. 

\subsection{Orthogonal and paraorthogonal polynomials}

Given a non-trivial probability measure $\mu$ on $\T$, i.e., one with infinite support, let $\Phi_n$ be the unique monic polynomial of degree $n$ satisfying the orthogonality relation
\[
	\int_{0}^{2\pi} \Phi_n(\ee^{\ii\theta}) \ee^{-\ii k \theta} \,\dd\mu(\ee^{\ii\theta}) = 0, \quad \mbox{for all } k=0,\ldots,n-1.
\]
These OPUC satisfy the Szeg\H{o} recursion relations:
\begin{align}
	\Phi_{n+1}(z)& =z\Phi_{n}(z)-\bar{\alpha}_{n}\Phi_{n}^{*}(z), \quad n\in\N_{0},
	\label{eq:recur_opuc}
	\\
	\Phi_{n+1}^{*}(z)& =\Phi_{n}^{*}(z)-z\alpha_{n}\Phi_{n}(z), \quad n\in\N_{0},
	\label{eq:recur_opuc_star}
\end{align}
where $\Phi_{n}^{*}(z)=z^{n}\overline{\Phi_{n}(1/\bar{z})}$, and $\{\alpha_{n}\}_{n=0}^{\infty}\subset\D^{\infty}$ are the Verblunsky coefficients.
The Verblunsky Theorem states that there is a one-to-one correspondence between all non-trivial probability measures on $\T$  and all sequences $\{\alpha_{n}\}_{n=0}^{\infty}\subset\D^{\infty}$.
The zeros of $\Phi_{n}$ are known to be located in~$\D$ for any non-trivial $\mu$ on $\T$.

If the Verblunsky coefficient $\alpha_{n-1}$ is replaced by a unimodular parameter $\beta\in \T$,
one obtains the POPUC:
\begin{equation}\label{eq:POPUC}
	\Phi^{(\beta)}_{n}(z)=z\Phi_{n-1}(z)-\bar{\beta}\Phi_{n-1}^{*}(z).
\end{equation}
A crucial property for our analysis is that all zeros of POPUC lie on~$\T$ and are simple.

\subsection{CMV matrices}\label{subsec:CMV}

From the spectral theory point of view, the probability measure $\mu$ associated with the sequence of Verblunsky coefficients $\{\alpha_{n}\}_{n=0}^{\infty}\subset\D^{\infty}$ is the spectral measure with respect to the cyclic vector $e_0$ of the unitary operator on $\ell^{2}(\N_{0})$, referred to as the CMV matrix:
\begin{equation}\label{eq:CMV}
	\mathcal{C}=\begin{pmatrix} \bar{\alpha}_{0} & \bar{\alpha}_{1}\rho_{0} & \rho_{1}\rho_{0} & 0 & 0 & \dots\\
		\rho_{0} & -\bar{\alpha}_{1}\alpha_{0} & -\rho_{1}\alpha_{0} & 0 & 0 & \dots\\
		0 & \bar{\alpha}_{2}\rho_{1} & -\bar{\alpha}_{2}\alpha_{1} & \bar{\alpha}_{3}\rho_{2} & \rho_{3}\rho_{2} & \dots\\
		0 & \rho_{2}\rho_{1} & -\rho_{2}\alpha_{1} & -\bar{\alpha}_{3}\alpha_{2} & -\rho_{3}\alpha_{2} & \dots\\
		0 & 0 & 0 & \bar{\alpha}_{4}\rho_{3} & -\bar{\alpha}_{4}\alpha_{3} & \dots\\
		\vdots & \vdots & \vdots & \vdots & \vdots & \ddots
	\end{pmatrix},
\end{equation}
where $\rho_{n}=\sqrt{1-|\alpha_{n}|^{2}}$. 

Given a CMV matrix $\mathcal{C}$ associated with Verblunsky coefficients $\{\alpha_n\}_{n=0}^\infty$, the once-stripped CMV matrix, denoted by $\mathcal{C}^{(1)}$, is the CMV matrix constructed from the shifted sequence of Verblunsky coefficients $\{\alpha_n\}_{n=1}^\infty$; that is, the sequence obtained by removing $\alpha_0$ and reindexing. Similar definition applies to $\mathcal{C}^{(j)}$, which is the CMV matrix constructed from $\{\alpha_n\}_{n=j}^\infty$, for any $j\in\N_0$.

Let us also introduce the cut-off CMV matrix $\mathcal{C}_{n}$,
which is the top-left $n\times n$ submatrix of matrix~\eqref{eq:CMV}. 
It is easy to show that $\Phi_{n}$ is the characteristic polynomial of~$\mathcal{C}_{n}$, i.e.
\begin{equation}\label{eq:charact}
	\Phi_{n}(z)=\det(z-\mathcal{C}_{n}), \quad n\in\N.
\end{equation}

An inspection of the structure of the CMV matrix shows that $\mathcal{C}_{n}$ depends only on the first $n$ Verblunsky coefficients. When needed, we will indicate this dependence  by writing
\[
\mathcal{C}_{n}=\mathcal{C}_{n}(\alpha_{0},\dots,\alpha_{n-1}).
\]
The POPUC~\eqref{eq:POPUC} therefore satisfy
\begin{equation}\label{eq:POPUCdef}
	\Phi_{n}^{(\beta)}(z)=\det\big(z-\mathcal{C}_{n}^{(\beta)}\big),
\end{equation}
where, for convenience, we abbreviate
\begin{equation}
	\mathcal{C}_{n}^{(\beta)}=\mathcal{C}_{n}(\alpha_{0},\dots,\alpha_{n-2},\beta).
\label{eq:CMV_beta_matrix_abbrev}
\end{equation}
Note that, unlike $\mathcal{C}_{n}$, the matrix $\mathcal{C}_{n}^{(\beta)}$ is unitary for any $\beta\in\T$, which explains why the zeros of POPUC are located on~$\T$. Moreover, all the zeros are simple.

\subsection{Carath{\' e}odory, Schur, and $m$-functions}

For a probability measure $\mu$ on $\T$, define its Carath{\' e}odory function by
\begin{equation}\label{eq:Cara}
	F(z)=\int_{0}^{2\pi}\frac{\ee^{\ii\theta}+z}{\ee^{\ii\theta}-z}\dd{\mu(\ee^{\ii\theta})} =
	\langle e_{0},(\mathcal{C}+z)(\mathcal{C}-z)^{-1}e_{0}\rangle , \quad z\in\D.
\end{equation}
This is an analytic function on $\D$ with $\Re F>0$ and $F(0)=1$. Using~\eqref{eq:Cara}, one may also define $F$ on $\C\setminus\supp\,\mu$, in which case it is straightforward to verify that
\begin{equation}\label{eq:Cara_symmetry}
	\overline{F(\bar{z}^{-1})} = -F(z), \quad z\in \C\setminus\supp\,\mu.
\end{equation}
We stress that $F$ on $\C\setminus\overline\D$ need not coincide with the analytic continuation  of $F$ on $\D$ through $\T$ if $\supp\,\mu = \T$ (see Lemma~\ref{lem:m_and_G}\textit{(i)} below for an example).

The Schur function of $\mu$ is defined via
\begin{equation}\label{eq:Schur}
	f(z) = \frac{1}{z}\frac{F(z)-1}{F(z)+1}. 
\end{equation}
The case $z=0$  in the last identity (as well as in~\eqref{eq:mSymm}, \eqref{eq:mGeronimus}, and \eqref{eq:fGeronimus} below) is understood by taking limits. 

{The} Schur function is analytic $\D$  and satisfies 
$\sup_{z\in\D}|f(z)| \le 1$. It will be convenient for us to extend~\eqref{eq:Schur} (and so also \eqref{eq:m_Schur} below) to  $z\in \C\setminus\supp\,\mu$, which requires allowing $f(z)$ to take value $\infty$ at those points $z\in\C\setminus\overline\D$ where $F(z)=-1$. 
For instance, under this convention, for the free CMV matrix ($\alpha_n\equiv 0$) we obtain $f(z)=0$ for all $z\in\D$ and $f(z)=\infty$ for all $z\in\C\setminus\overline{\D}$, see Section~\ref{subsec:constant}.
From~\eqref{eq:Cara_symmetry}, we obtain the symmetry
\begin{equation}\label{eq:symSchur}
	\overline{f(\bar{z}^{-1})}  = \frac{1}{f(z)}, \quad z\in \C\setminus\supp\,\mu.
\end{equation}

Finally, we introduce a notation for the $(0,0)$-entry of the resolvent of $\mathcal{C}$:
\begin{equation}\label{eq:def_func_m}
	m(z)= \int_0^{2\pi}\frac{{\dd{\mu(\ee^{\ii\theta})}}}{z-\ee^{\ii\theta}}=\langle e_{0},(z-\mathcal{C})^{-1}e_{0}\rangle,
	\quad z\in\C\setminus\supp\,\mu.
\end{equation}

We summarize further properties and interrelations between the three introduced functions in the next lemma that will be needed later.

\begin{lem}\label{lem:m}\hfill
	\begin{enumerate}[(i)]
		\item  For all $z\in\C\setminus\supp\,\mu$, the following equalities hold:
		\begin{align}
			\label{eq:m_Cara}
			& F(z) = 1-2z m(z),
			\\
			\label{eq:m_Schur}
			&f(z) = \frac{m(z)}{zm(z)-1} =  -\frac{zm(z)}{\overline{m(\bar{z}^{-1})}},
			\\
			\label{eq:mSymm}
			&\overline{m(\bar{z}^{-1})}  = z(1-z m(z)).
		\end{align}
		
		\item Zero sets of $m(z)$ and $f(z)$ on $\C\setminus\supp\mu$ coincide. In particular, $m(z) \ne 0$  for all $z\in\C\setminus\overline\D$ and for $z\in\T\setminus\supp\mu$.
	\end{enumerate}
\end{lem}
\begin{proof}
	Formulas \eqref{eq:m_Cara}--\eqref{eq:mSymm} are immediate from the definitions. Next, since $f$ does not take infinite values in $\overline\D\setminus\supp\mu$, it follows from~\eqref{eq:symSchur} that $f(z)\ne 0$ for $z\in\C\setminus\overline\D$ as well as for $z\in\T\setminus\supp\mu$. Combining this with the first equality in~\eqref{eq:m_Schur} implies claim {\it (ii)}.
\end{proof}

\subsection{Constant Verblunsky coefficients}\label{subsec:constant}

In this section, we assume that all the Verblunsky coefficients are constant, that is, $\alpha_n= \alpha$ for all $n\in\N_0$, with $\alpha\in\D$. 
The corresponding CMV operator~\eqref{eq:CMV} and the $n\times n$ cut-off CMV matrix will be denoted by $\mathcal{C}(\alpha)$ and $\mathcal{C}_n(\alpha)$, respectively.

In the so-called free case, when $\alpha=0$, we have $\Phi_n(z)=z^n$ and the corresponding measure of orthogonality is the Lebesgue measure on $\T$. For $\alpha\in\D\setminus\{0\}$, the orthogonal polynomials $\Phi_n$ are known as the Geronimus polynomials, see \cite{geronimus_vkgu66}, as well as \cite[Sec.~2]{golinski_maa99} or \cite[Ex.~1.6.12]{simon_opuc1}.

We will also allow  $\alpha\in\T$ in the CMV matrix~\eqref{eq:CMV}. In this case, 
\[
\mathcal{C}(\alpha)=\diag(\bar{\alpha},-1,-1,\dots),
\]
where $\diag$ denotes a diagonal matrix. The spectral measure $\mu$ of~$\mathcal{C}(\alpha)$ with respect to $e_0$ is the Dirac delta measure $\delta_{\bar\alpha}$. This measure is trivial, and therefore it does not uniquely determine its orthogonal polynomials. Nevertheless, all the functions discussed 
below are {unambiguously} defined and consistent with this setting $\alpha\in\T$,  $\mu = \delta_{\bar\alpha}$. 

Next we denote $a=|\alpha|\in [0,1]$, 
\begin{equation}
	\theta_{a}=2\arcsin a
	\label{eq:def_theta},
\end{equation}
and
\begin{equation}
	\label{eq:arc}
	\Gamma_a=\left\{\ee^{\ii\theta}\in \T \;\big|\; \theta_{a}\le \theta \le 2\pi-\theta_{a}\right\}. 
\end{equation}
Let $\Gamma_a^\circ$ be the interior of $\Gamma_a$ in the induced topology of $\T$, in particular $\Gamma_1^\circ = \emptyset$. 
It can be shown, see e.g., \cite[Thm.~1.6.13]{simon_opuc1} that the spectrum of  $\mathcal{C}(\alpha)$ is 
\begin{equation*}
	\sigma_\alpha =
	\begin{cases}
		\Gamma_a, & \mbox{ if } |\alpha+\tfrac12|\le \tfrac12,
		\\[2pt]
		\Gamma_a\cup\{\frac{1+\bar\alpha}{1+\alpha}\}, & \mbox{ if } |\alpha+\tfrac12| > \tfrac12. 
	\end{cases}
\end{equation*}
Note that this remains valid even if $\alpha\in\T$, since in that case {$(1+\bar\alpha)/(1+\alpha) = \bar\alpha$}.

Throughout this paper, we use the square root function
\begin{equation}\label{eq:sqrt}
	\sqrt{\left(z-\ee^{\ii\theta_{a}}\right)\left(z-\ee^{-\ii\theta_{a}}\right)} = \sqrt{(z-1)^{2}+4z a^{2}} ,
\end{equation}
which we take to be analytic on $\C\setminus\Gamma_a$ with a branch cut along $\Gamma_a^\circ$, and fixed by
\[
	\lim_{z\to0}\sqrt{\left(z-\ee^{\ii\theta_{a}}\right)\left(z-\ee^{-\ii\theta_{a}}\right)} = 1.
\]
If $a=0$, we choose~\eqref{eq:sqrt} to be interpreted as $1-z$ if $z\in\D$ and $z-1$ if $z\in\C\setminus\overline{\D}$. If $a=1$ then the expression~\eqref{eq:sqrt} is treated as $z+1$ without any branch cut.

With this choice of the square root, we define the function $G_a$, analytic on $\C\setminus \Gamma_a$,  by
\begin{equation} \label{eq:def_G}
	G_{a}(z)=\frac{1}{2}\left(z+1+\sqrt{\left(z-\ee^{\ii\theta_{a}}\right)\left(z-\ee^{-\ii\theta_{a}}\right)}\right).
\end{equation}
This function naturally appears as the limit of the ratio of orthogonal polynomials, see Theorem~\ref{thm:ratio} below, and is closely related to the logarithmic potential of the equilibrium measure of $\Gamma_{a}$, see~\eqref{eq:U_eq_log_G} below.

It is easy to verify that, for $0<a\le 1$, function $G_a$ satisfies
\begin{align}
	& \overline{G_a(\bar{z}^{-1})}   = z^{-1} G_a(z), \label{eq:GSymm} 
	\\
	& (z-G_a(z))(1-G_a(z)) = {a^{2}} z \label{eq:GProperty}
\end{align}
for $z\in\C\setminus\Gamma_a$.

We denote by $F_\alpha$, $f_\alpha$, and $m_\alpha$, the Carath\'{e}odory function~\eqref{eq:Cara}, the Schur function~\eqref{eq:Schur}, and the $m$-function~\eqref{eq:def_func_m}, respectively, in the case of constant coefficients $\alpha_n\equiv \alpha$. The next lemma  summarizes basic formulas for  $F_\alpha$, $f_\alpha$, and $m_\alpha$ in forms suitable for later purposes.

\begin{lem}\label{lem:m_and_G}\leavevmode
	\begin{enumerate}[(i)]
		\item For $\alpha=0$, we have 
		\begin{alignat*}{2}
			&	G_0(z)  =
			\begin{cases}
				1, & \mbox{ if } z\in\D, \\
				z, & \mbox{ if } z\in \C\setminus\overline\D,
			\end{cases}
			\qquad
			&F_0(z)=
			\begin{cases}
				1, & \mbox{ if } z\in\D, \\
				-1, & \mbox{ if } z\in \C\setminus\overline\D,
			\end{cases}
			\\
			& m_0(z)  =
			\begin{cases}
				0, & \mbox{ if } z\in\D, \\
				\frac1z, & \mbox{ if } z\in \C\setminus\overline\D,
			\end{cases}
			\qquad
			& f_0(z)  =
			\begin{cases}
				0, & \mbox{ if } z\in\D, \\
				\infty, & \mbox{ if } z\in \C\setminus\overline\D.
			\end{cases}
		\end{alignat*}
		\item For $\alpha\in\T$ and $z\in\C\setminus {\{\bar\alpha\}}$, we have
		$$
		G_1(z) = z+1, \quad 
		F_\alpha(z)={\frac{\bar\alpha+z}{\bar\alpha-z}},
		\quad
		m_\alpha(z) = \frac{1}{z-\bar\alpha},
		\quad
		f_\alpha(z) = \alpha.
		$$
		\item For $\alpha\in\D\setminus\{0\}$ and $z\in\C\setminus\sigma_\alpha$, we have
		\begin{align}
     		\label{eq:FGeronimus}
			F_\alpha(z) & 
			= 1+ \frac{2(G_a(z)+\alpha z -1)}{1+\bar\alpha-(1+\alpha)z},
			\\
			\label{eq:mGeronimus}
			m_\alpha(z) & =
			\frac{1}{z}\frac{G_a(z)+\alpha z -1}{(1+\alpha)z-(1+\bar\alpha)},
			\\
			\label{eq:fGeronimus}
			f_\alpha(z) & = \frac{G_a(z) -1}{\bar\alpha z}, 
		\end{align}
		where $G_a$ is given by~\eqref{eq:def_G}.
	\end{enumerate}
\end{lem}
\begin{proof}
The formulas in \textit{(i)} and \textit{(ii)} are elementary from the definitions.
	
We prove part~\textit{(iii)}. The expression \eqref{eq:FGeronimus} for $F_\alpha$ can  be found in~\cite[Eq.~(2.5)]{golinski_maa99} (note that our $\alpha$ and $G_a$ correspond to their $-\bar{a}$ and $z_1$, respectively). Combining~\eqref{eq:FGeronimus} with~\eqref{eq:def_func_m} implies~\eqref{eq:mGeronimus}. Finally, expression~\eqref{eq:fGeronimus} appears in~\cite[Eq.~(1.6.82)]{simon_opuc1}. Alternatively, one can deduce~\eqref{eq:fGeronimus} by employing~\eqref{eq:Schur} and~\eqref{eq:GProperty}.
\end{proof}

\subsection{Periodic OPUC}\label{ss:periodic}

Suppose now that $\{\alpha_{n}\}_{n=0}^{\infty}\in\D^\infty$ is a $p$-periodic sequence for some $p\in\N$, i.e., $\alpha_{k+p}=\alpha_{k}$, for all $k\in\N_{0}$. An important role is played by the discriminant
\begin{equation}\label{eq:def_discr}
	\Delta(z)=\Delta\left(\alpha_{0},\dots,\alpha_{p-1};z\right)=z^{-p/2}\Tr\left(A(\alpha_{p-1},z)\dots A(\alpha_{0},z)\right),
\end{equation}
where $z\in\C\setminus\{0\}$ and
\[
A(\alpha,z)=\frac{1}{\sqrt{1-|\alpha|^{2}}}\begin{pmatrix}
	z & -\bar{\alpha} \\
	-\alpha z & 1
\end{pmatrix}.
\]
To simplify technical difficulties arising from the fractional power of $z$ in the definition~\eqref{eq:def_discr}, it is often convenient to assume that $p$ is even. This causes no loss of generality, since any $p$-periodic sequence is also $2p$-periodic; cf.~\cite[Thm.~11.1.3]{simon_opuc2} and Remark~1 therein.

The discriminant $\Delta$ is real on $\T$, satisfies 
\begin{equation}\label{eq:Delta(0)}
	\lim_{z\to0}z^{p/2}\Delta(z)=\prod_{j=0}^{p-1} \rho_j^{-1},
\end{equation}
and
\[
	\mathcal{B}_\Delta=\Delta^{-1}([-2,2])
\]
is a so-called finite-gap set, i.e., a union of $p$ closed circular arcs on~$\T$
that can touch at endpoints but are otherwise disjoint. The essential support of the spectral measure $\mu$ coincides with $\mathcal{B}_\Delta$.

In particular, the case of constant Verblunsky coefficients $\alpha_n\equiv \alpha\in\D$ from the previous section corresponds to
$$
\Delta(\alpha,\alpha;z) =\frac{1}{1-|\alpha|^2} ( z+z^{-1} + 2|\alpha|^2),
$$
and $\mathcal{B}_\Delta$ becomes just $\Gamma_{|\alpha|}$ as in~\eqref{eq:arc} if $|\alpha|<1$.

\subsection{Elements of potential theory}\label{sect:potential}

We recall a few facts from potential theory. These facts can be deduced, for instance, from \cite[Chap.~11]{simon_opuc2} where they are summarized in greater generality. The reader may also consult~\cite{saff-totik_97,stahl-totik_92}.
For $a\in[0,1]$, we denote by~$\nu_{a}$ the equilibrium measure of the circular arc $\Gamma_{a}$ in~\eqref{eq:arc}, where $\theta_{a}$ is as in~\eqref{eq:def_theta}. If $a\in[0,1)$, the probability measure~$\nu_{a}$ is purely absolutely continuous with the density
\begin{equation}
	\frac{\dd\nu_{a}}{\dd\theta}{(\ee^{\ii\theta})}=\frac{1}{2\pi}\frac{\sin(\theta/2)}{\sqrt{\cos^{2}(\theta_{a}/2)-\cos^{2}(\theta/2)}}, \quad \ee^{\ii\theta}\in \Gamma_a. 
	\label{eq:meas_nu}
\end{equation}
If $a=1$, $\nu_{1}$ is the Dirac unit mass measure $\delta_{-1}$ supported on {the one-point set} $\Gamma_{1}=\{-1\}$. The logarithmic potential of~$\nu_{a}$ is
\begin{equation} \label{eq:U_eq_log_G}
	U_{\nu_{a}}(z)=-\int_{\theta_{a}}^{2\pi-\theta_{a}}\log\big|z-\ee^{\ii\theta}\big|\dd\nu_{a}\big( \ee^{\ii \theta} \big)=-\log|G_{a}(z)|, \quad z\in\C\setminus\T,
\end{equation}
where $G_a$ is {as in}~\eqref{eq:def_G}.

More generally, given a discriminant $\Delta$ as in~\eqref{eq:def_discr}, denote by $\nu_{\Delta}$  the equilibrium measure of  the corresponding finite-gap set $\mathcal{B}_\Delta$. It is well known that $\nu_{\Delta}$ is supported on  $\mathcal{B}_{{\Delta}}$ and is absolutely continuous with the density\footnote[1]{Be aware of a typo in the density formulas in \cite[Eqs.~(11.1.18)--(11.1.19)]{simon_opuc2}: the expression includes a superfluous factor $1/2$.}
\begin{equation}
	\frac{\dd \nu_{\Delta}}{\dd\theta}(\ee^{\ii\theta})=\frac{1}{\pi p}\frac{\left|\Delta'(\ee^{\ii\theta})\right|}{\sqrt{4-\Delta^{2}(\ee^{\ii\theta})}}.
	\label{eq:def_nu_periodic}
\end{equation}
The prime designates differentiation of $\Delta(\ee^{\ii\theta})$ with respect to $\theta$.

For $z\in\C\setminus \mathcal{B}_\Delta$, the logarithmic potential of $\nu_{\Delta}$ reads
\begin{equation}	\label{eq:U_eq_log_G_Delta}
	U_{\nu_{\Delta}}(z)=-\int_{\mathcal{B}_\Delta}\log\big|z-\ee^{\ii\theta}\big|\dd\nu_{\Delta}\big(\ee^{\ii \theta} \big)=- \frac{1}{p}\log\left|G_{\Delta}(z)\right|,
\end{equation}
where
\begin{equation}	\label{eq:def_G_Delta}
	G_{\Delta}(z)=\frac{z^{p/2}}{2}\left(\prod_{j=0}^{p-1}\sqrt{1-|\alpha_{j}|^2}\right)\left(\Delta(z)+\sqrt{\Delta^{2}(z)-4}\right).
\end{equation}
The branch of the square root in~\eqref{eq:def_G_Delta} is to be chosen such that it maximizes~$|\Delta(z)+\sqrt{\Delta^{2}(z)-4}|$. 
For a future reference, note that 
\begin{equation}
	\label{eq:ptyG1}
	G_{\Delta}(0)=1,
\end{equation}
which follows from ~\eqref{eq:Delta(0)} and \eqref{eq:def_G_Delta}.


\subsection{Zero counting measure and root asymptotics}\label{sec:root_asym}

Let $\mathcal{C}^{(0)}$ be a CMV matrix with either constant or periodic Verblunsky coefficients as in Section~\ref{subsec:constant} or ~\ref{ss:periodic}, respectively. Let $\Delta$, $\mathcal{B}_\Delta$, and $\nu_\Delta$ denote the discriminant, the associated finite-gap set, and the equilibrium measure, respectively. 

Let $\mathcal{C}(\alpha_0,\alpha_1,\ldots)$ be another CMV matrix that is a compact perturbation of $\mathcal{C}^{(0)}$. In other words, $\alpha_n$ either converges to some $\alpha$ or is asymptotically $p$-periodic. 
Let $\Phi_n(z)$ and  $\Phi^{(\beta)}_n(z)$ be the corresponding OPUC and POPUC, see~\eqref{eq:charact} and~\eqref{eq:POPUCdef}. Denote by $z_{1,n},\dots,z_{n,n}$ and $z^{(\beta)}_{1,n},\dots,z^{(\beta)}_{n,n}$ the zeros of $\Phi_{n}$ and $\Phi_{n}^{(\beta)}$, respectively. Recall the definition of 
the zero-counting measures of $\Phi_n$ and $\Phi_n^{(\beta)}$,
\[
	\nu_{n}=\frac{1}{n}\sum_{j=1}^{n}\delta_{z_{j,n}}
\quad\mbox{ and }\quad
	\nu^{(\beta)}_{n}=\frac{1}{n}\sum_{j=1}^{n}\delta_{z^{(\beta)}_{j,n}}.
\]

It is well known that for the case when $\mathcal{C}^{(0)}$ is \textbf{not} the free CMV matrix, both zero-counting measures converge to the corresponding equilibrium meausure, that is,
\begin{equation}
	\wlim_{n\to\infty} \nu_{n} = 
	\wlim_{n\to\infty} \nu_{n}^{(\beta)}=\nu_{\Delta},
	\label{eq:nu_n_tends_to_nu_scalar_opuc}
\end{equation}
see~\cite{stahl-totik_92,simon_opuc1} and references therein.

For the case when $\mathcal{C}^{(0)}$ \textbf{is} the free CMV matrix, i.e., $\alpha_n\to0$, we still have $\nu_{n}^{(\beta)}\to\nu_{0}$ weakly as $n\to\infty$, where $\nu_0$ the uniform probability measure on $\T$. However, things are more subtle for $\nu_n$. If one additionally assumes that
\[
	\lim_{n\to\infty} |\alpha_n|^{1/n} = A\in[0,1],
\]
then Mhaskar and Saff~\cite{mha-saf_jat90} showed that $\nu_{n}$ converges weakly to the uniform probability measure on the circle of radius $A$ {centered at} the origin.

Convergence in~\eqref{eq:nu_n_tends_to_nu_scalar_opuc} is closely related to the so-called {\it root asymptotics} of OPUC: 
\begin{equation}\label{eq:root}
	\log|\Phi_n(z)|^{1/n} = \int \log|z-\xi|\,  \dd \nu_{n} (\xi)   \to \int \log|z-\xi|\,  \dd\nu_\Delta(\xi) 
	=
	-U_{\nu_\Delta}(z)
\end{equation} 
for $z\in\C\setminus\overline\D$ as $n\to\infty$. Root asymptotics~\eqref{eq:root} holds under much weaker assumptions than we are assuming here. Indeed, 
in the current setting one can establish the stronger {\it ratio asymptotics} of orthogonal polynomials. This is the topic of our next section.

\section{Ratio asymptotics for varying POPUC and OPUC}\label{sec:ratio_asym}

The purpose of this section is to review the known results on ratio asymptotics for OPUC and POPUC and to present a short proof that also applies in the context of \textbf{variable orthogonality}, meaning that each $N\in\N$ may correspond to a different orthogonality measure $\mu_N$.
Denote the Verblunsky coefficients of $\mu_N$ by $\{\alpha_{n,N}\}_{n=0}^\infty$ and the corresponding  OPUC and POPUC by
\[
	\Phi_{n,N}(z)=\det\left(z-\mathcal{C}_{n}(\alpha_{0,N},\alpha_{1,N},\dots,\alpha_{n-1,N})\right)
\]
and
\[
	\Phi_{n,N}^{(\beta)}(z)=\det\left(z-\mathcal{C}_{n}(\alpha_{0,N},\alpha_{1,N},\dots,\alpha_{n-2,N},\beta)\right),
\]
where $\beta\in\T$.

\subsection{Asymptotically constant Verblunsky coefficients}\label{subsec:RatioConstant}

We assume that the Verblunsky coefficients of $\mu_N$ satisfy 
\begin{equation}\label{eq:convergence}
	\lim_{n/N\to t}\alpha_{n,N}=\alpha.
\end{equation}
for some $\alpha\in\overline{\D}$ {and $t>0$}. If $\mu_N$ does not depend on $N$, then ~\eqref{eq:convergence} is typically called the L\'{o}pez or Nevai condition.
It is then well known  that this condition implies ratio asymptotics of orthogonal polynomials: for $\alpha=0$ this is quite basic, see~\cite[Sec.~1.7]{simon_opuc1}, while the cases $\alpha\in\overline{\D}$ are treated in~\cite{BarLop,Khr02} and~\cite[Sec.~9.5]{simon_opuc2}; see also~\cite{BRZ} for further developments.

The case of varying orthogonality is stated in Theorem~\ref{thm:ratio} below. We provide a~short and simple proof based on the operator theory of CMV matrices. It is related to the ideas of~\cite{Khr02} (and~\cite[Sec.~9.5]{simon_opuc2}), but instead of Schur functions we work with $m$-functions (an approach already used for Jacobi matrices in~\cite{Sim04}). 
This choice is natural from the CMV perspective, since the ratio of polynomials can be expressed as a single matrix element of the resolvent of the CMV matrix via Cramer's rule. 

The proof of Theorem~\ref{thm:ratio} takes an especially streamlined form if one is satisfied with convergence on compacts of $\D$ (or of $\C\setminus\overline{\D}$ in part {\it(iv)}). With just a bit more care and an application of the Vitali theorem (or Vitali--Porter theorem, see, for example,~\cite[Sec.~2.4]{schiff_93}), we are able to prove convergence on compacts of $\C$ as long as one stays away from the limiting zero set of the polynomials.
Furthermore, the argument extends to the variable orthogonality setting with no additional complications. 

The idea of the proof is as follows. We represent the ratio $\Phi_n/\Phi_{n+1}^{(\beta)}$  as the $(n,n)$-entry of the resolvent of $\mathcal{C}^{(\beta)}_{n+1}$. Equivalently, this is the $(0,0)$-entry of the resolvent of the CMV matrix but flipped ``on its head'', which has itself a CMV  structure with Verblunsky coefficients reversed and conjugated (see Lemma~\ref{lem:relf_id}). Under the condition~\eqref{eq:convergence} this becomes the $(0,0)$-entry of the resolvent of a varying CMV matrix whose Verblunsky coefficients all converge to the same value in the limit $n/N\to t$. But this is the $m$-function $m_\alpha$ from Section~\ref{subsec:constant}. After we establish the asymptotics of $\Phi_n/\Phi_{n+1}^{(\beta)}$,  all other {limit relations will} follow from the Szeg\H{o} recurrences.

Below, $\delta_{m,n}$ denotes the Kronecker delta symbol and $A^{T}$ stands for the transpose of a matrix $A$.

\begin{lem}\label{lem:relf_id}
	Let $n\in\N$, $\beta\in\T$. Introduce two $n\times n$ unitary matrices
	\begin{equation*}
		\left(Q_{n}\right)_{j,k} =\delta_{j,n-1-k}, \quad \mbox{ for } j,k\in\{0,1,\dots,n-1\}
	\end{equation*}
	and 
	$P_{n}(\beta) =\operatorname{diag}(1,\beta,1,\beta,\ldots)$.
	Then
	\[
	Q_n\mathcal{C}_{n}(\alpha_{0},\dots,\alpha_{n-2},\beta) Q_n=\begin{cases}
		P_{n}(\beta)\mathcal{C}_{n}^{T}(-\bar{\alpha}_{n-2}\beta,\dots,-\bar{\alpha}_{0}\beta,\beta) P_{n}^{*}(\beta), & \mbox{ if } n \mbox{ is odd,}\\[2pt]
		P_{n}^{*}(\beta)\mathcal{C}_{n}(-\bar{\alpha}_{n-2}\beta,\dots,-\bar{\alpha}_{0}\beta,\beta)P_n(\beta), & \mbox{ if } n \mbox{ is even.}
	\end{cases}
	\]
\end{lem}

\begin{proof}
	See~\cite[Lem.~5.1 and Sec.~7]{KK}.
\end{proof}

For the formulation of the ratio asymptotic formulas, we need to introduce the sets of limit points of zeros of POPUC. Let $\Lambda_{\beta}$ denote the set of (weak) limit points of zeros of $\Phi_{n,N}^{(\beta)}$, i.e.
\begin{equation}\label{eq:Lambda1}
	\Lambda_{\beta,t} = \Big\{\ee^{\ii\theta} \in \T \;\Big|\; \liminf_{j\to\infty} \dist\left(\ee^{\ii\theta},\calZ(\Phi_{n_j,N_j}^{(\beta)})\right) =0 \mbox{ for some } n_j,N_j \mbox{ as in~\eqref{eq:subs}}\Big\},
\end{equation}
where 
\[
\calZ(\Phi_{n,N}^{(\beta)}) = \big\{z\in\T \;\big|\; \Phi_{n,N}^{(\beta)}(z)=0 \big\}.
\]
We also define
\begin{equation}\label{eq:Lambda2}
	\Lambda_t = \bigcap_{\beta\in\T} 	\Lambda_{\beta,t}.
\end{equation}

\begin{rem}\label{rem:spec_incl_lam}
	It is an interesting problem to understand the sets $\Lambda_{\beta,t}\subseteq\T$ and $\Lambda_t\subseteq\T$ which we will not attempt to address here beyond noticing that, under the condition~\eqref{eq:convergence}, $\Lambda_{\beta,t}\supseteq \sigma_{-\bar{\alpha}\beta} \supseteq \Gamma_a$, and therefore $\Lambda_t\supseteq \Gamma_a$, which follows from the proof {of Theorem~\ref{thm:ratio}}\emph{(i)} below. 
\end{rem}


\begin{thm}\label{thm:ratio}
	Let $\beta\in\T$, $t>0$, and let $\{\alpha_{n,N} \mid n,N\in\N_{0}\}\subset\D$ satisfy~\eqref{eq:convergence} with $\alpha\in\overline{\D}$. {Then the following limits hold:}
	\begin{enumerate}[(i)]
		\item 
		\[
		\lim_{n/N\to t}\frac{\Phi_{n,N}(z)}{\Phi_{n+1,N}^{(\beta)}(z)}= m_{-\bar{\alpha}\beta}(z) 
		\]
		uniformly on compacts of $\C\setminus{\Lambda_{\beta,t}}$,
		\item 
		\[
		\lim_{n/N\to t}\frac{\Phi_{n,N}(z)}{\Phi_{n,N}^{*}(z)}= f_{-\bar{\alpha}}(z)
		\]
		uniformly on compacts of $\C\setminus{\Lambda_t}$ if $\alpha\neq0$ and of $\D$ if $\alpha=0$,
		\item 
		\[
		\lim_{n/N\to t}\dfrac{\Phi^*_{n+1,N}(z)}{\Phi^*_{n,N}(z)}=G_a(z)
		\]
		uniformly on compacts of $\C\setminus{\Lambda_t}$ if $\alpha\neq0$ and uniformly on $\overline{\D}$ if $\alpha=0$\footnote[1]{\label{foot:1}We then extend $G_0$ to its non-tangential limits on $\T$ from the inside of $\D$ for {\it (iii)} and from the outside of $\D$ for {\it (iv)}.},
		\item 
		\[
		\lim_{n/N\to t}\dfrac{\Phi_{n+1,N}(z)}{\Phi_{n,N}(z)}=G_a(z) 
		\]
		uniformly on compacts of $\C\setminus{\Lambda_t}$ if $\alpha\neq0$ and of $\C\setminus\D$ if $\alpha=0${\normalfont \textsuperscript{\ref{foot:1}}},
		\item 
		\begin{equation}
			\label{eq:ratio_popuc_G}
			\lim_{n/N\to t}\frac{\Phi_{n+1,N}^{(\beta)}(z)}{\Phi_{n,N}^{(\beta)}(z)}=G_a(z) 
		\end{equation}
		uniformly on compacts of $ \C\setminus \Lambda_{\beta,t}$.
	\end{enumerate}
\end{thm}
\begin{proof}
	\textit{(i)} By a variant of Cramer's Rule, we can compute the $(n,n)$-entry of the inverse of $z-\mathcal{C}_{n+1}(\alpha_{0,N},\alpha_{1,N},\dots,\alpha_{n-1,N},\beta)$ via
	\begin{equation}
		\langle e_{n},\left(z-\mathcal{C}_{n+1}(\alpha_{0,N},\alpha_{1,N},\dots,\alpha_{n-1,N},\beta)\right)^{-1}e_{n}\rangle
		=
		\frac{\Phi_{n,N}(z)}{\Phi_{n+1,N}^{(\beta)}(z)}.
		\label{eq:start_id_ratio_asympt_inproof}
	\end{equation}
	We claim that this ratio is equal to the $m$-function ~\eqref{eq:def_func_m} of the finite CMV matrix $\mathcal{C}_{n+1}(-{\bar{\alpha}_{n-1,N}}\beta,-{\bar{\alpha}_{n-2,N}}\beta,\dots,-{\bar{\alpha}_{0,N}}\beta,\beta)$, which we henceforth denote by 
	$$
	\mathcal{C}_{n+1,N}=\mathcal{C}_{n+1}(-{\bar{\alpha}_{n-1,N}}\beta,-{\bar{\alpha}_{n-2,N}}\beta,\dots,-{\bar{\alpha}_{0,N}}\beta,\beta)
	.
	$$
	Indeed, taking into account identities  $e_n=Q_{n+1}e_0$ and $e_0=P_{n+1}(\beta)e_0=P_{n+1}^{*}(\beta)e_0$ (see Lemma~\ref{lem:relf_id}), the left hand side of~\eqref{eq:start_id_ratio_asympt_inproof} can be written as
	\[
	{\langle e_{0},\left(z-P_{n+1}(\beta)Q_{n+1}\mathcal{C}_{n+1}(\alpha_{0,N},\alpha_{1,N},\dots,\alpha_{n-1,N},\beta)Q_{n+1}P_{n+1}^{*}(\beta) \right)^{-1}e_{0}\rangle}.
	\]
	It then follows from Lemma~\ref{lem:relf_id} that~\eqref{eq:start_id_ratio_asympt_inproof} equals to 
	$ \langle e_{0},\left(z-\mathcal{C}_{n+1,N} \right)^{-1}e_{0}\rangle$,
	where we use that the transpose does not effect the value of this scalar product.
	
	Under condition~\eqref{eq:convergence}, the Verblunsky coefficients of $\mathcal{C}_{n+1,N}$ converge to those of $\mathcal{C}(-\bar{\alpha}\beta)$ as $n/N\to t$. This implies the weak convergence of the spectral measures, see, e.g.,~\cite[Thm.~1.5.6]{simon_opuc1}, and therefore the convergence of the $m$-functions~\eqref{eq:def_func_m} uniformly on compacts of $\C\setminus\T$. This yields claim~\textit{(i)} on $\C\setminus\T$. We next extend the claim to the full set $\C\setminus{\Lambda_{\beta,t}}$.
	
	Choose sequences $n_j,N_j\to\infty$ with $n_j/N_j\to t$, and let $K$ be a compact subset of $\C\setminus{\Lambda_{\beta,t}}$. By the definition of ${\Lambda_{\beta,t}}$ and Lemma~\ref{lem:relf_id}, there exists $\varepsilon>0$ such that the spectrum of $\mathcal{C}_{n_j+1,N_j}$ is $\varepsilon$-distance away from $K$ for all $j$ sufficiently large. It follows that the $m$-functions of $\mathcal{C}_{n_j+1,N_j}$ are bounded by the constant $1/\varepsilon$ uniformly on $K$ and $j$ sufficiently large. Applying Vitali's theorem on analytic convergence, we obtain uniform convergence on compacts of $\C\setminus {\Lambda_{\beta,t}}$. In particular, this shows that $m_{-\bar{\alpha}\beta}$ is analytic on $\C\setminus{\Lambda_{\beta,t}}$. Hence $\sigma_{-\bar{\alpha}\beta}\subseteq{\Lambda_{\beta,t}}$, proving the claims of Remark~\ref{rem:spec_incl_lam}.
	
	\smallskip
	
	\textit{(ii)} 
	First, recall that
	\begin{equation}\label{eq:POPUCSymmetry}
		\Phi_n^{(\beta)}(z) = -\bar\beta z^n\overline{\Phi_n^{(\beta)}(\bar{z}^{-1})},
	\end{equation}
	which follows from the facts that $\Phi_n^{(\beta)}$ has all the zeros on $\T$ and $\Phi_n^{(\beta)}(0)=-\bar{\beta}$.
	This together with \textit{(i)} gives
	\begin{equation}
		\label{eq:starPOPUCAs}
		\lim_{n/N\to t}\frac{\Phi^*_{n,N}(z)}{\Phi_{n+1,N}^{(\beta)}(z)}=
		\lim_{n/N\to t}\frac{\overline{\Phi_{n,N}(\bar{z}^{-1})}}{-\bar{\beta} z\overline{\Phi_{n+1,N}^{(\beta)}(\bar{z}^{-1})}}=  -\beta z^{-1} \overline{m_{-\bar{\alpha}\beta}(\bar{z}^{-1}) }
	\end{equation}
	uniformly on compacts of $ \C\setminus {\Lambda_{\beta,t}}$. Since the limit function is nonzero on $\D$ by Lemma~\ref{lem:m}{\it(ii)}, we obtain
	\begin{equation}\label{eq:proof1}
		\lim_{n/N\to t}\frac{\Phi_{n,N}(z)}{\Phi_{n,N}^{*}(z)}=
		\lim_{n/N\to t}\frac{\Phi_{n,N}(z)}{\Phi_{n+1,N}^{(\beta)}(z)}
		\frac{\Phi_{n,N+1}^{(\beta)}(z)}{\Phi_{n,N}^{*}(z)}
		=
		-\bar{\beta}z \frac{m_{-\bar{\alpha}\beta}(z)}{\overline{m_{-\bar{\alpha}\beta}(\bar{z}^{-1}) }}
	\end{equation}
	uniformly on compacts of $\D$. The limit function simplifies to $\bar{\beta}f_{-\bar{\alpha}\beta}(z)$ by~\eqref{eq:m_Schur}. As $\beta\in\T$ is arbitrary, the limiting function in~\eqref{eq:proof1} must be $\beta$-independent, so we may put $\beta=1$ to see that the limit is $f_{-\bar\alpha}(z)$.
	
	Suppose now additionally that $\alpha\ne 0$. Then the limit function in~\eqref{eq:starPOPUCAs} is nonzero on  $\C\setminus {\Lambda_{\beta,t}}$: indeed, $m_{-\bar{\alpha}\beta}(z)=0$ forces $f_{-\bar{\alpha}\beta}(z)=0$ by~\eqref{eq:m_Schur}, which in turn gives $G_{a}(z)=1$ by~\eqref{eq:fGeronimus},  contradicting~\eqref{eq:GProperty}.
	The argument in~\eqref{eq:proof1} then shows uniform convergence on compacts of $\C\setminus {\Lambda_{\beta,t}}$ for all $\beta\in\T$. It follows that the limit holds locally uniformly on
	the union of sets $\C\setminus\Lambda_{\beta,t}$, where $\beta$ ranges $\T$, i.e., on the set $\C\setminus\Lambda_{t}$.
	
	\smallskip
	
	\textit{(iii)} We rewrite Szeg\H{o}'s recurrence~\eqref{eq:recur_opuc_star} as
	\[
	\frac{\Phi^*_{n+1,N}(z)}{\Phi^*_{n,N}(z)}=1-z\alpha_{n,N}\frac{\Phi_{n,N}(z)}{\Phi^{*}_{n,N}(z)}
	\]
	and apply claim~\textit{(ii)}. If $\alpha=0$ then we restrict ourselves to $z\in\overline{\D}$ and use that $|\Phi_{n,N}/\Phi^{*}_{n,N}|\leq 1$ (since this is a Blaschke product) to get the limit $1$ uniformly on $\overline{\D}$. If $\alpha\ne 0$ then for any $z\in\C\setminus{\Lambda_{t}}$ we obtain {the limit} $1-z\alpha f_{-\bar\alpha}(z)$ {which coincides with $G_a(z)$} by Lemma~\ref{lem:m_and_G}\emph{(ii)}--\emph{(iii)}. 
	
	\smallskip
	
	\textit{(iv)} It follows readily from claim~\textit{(iii)}, identity $\Phi_n^*(z) = z^n \overline{\Phi_n(\bar{z}^{-1})}$, and~\eqref{eq:GSymm}.

	\smallskip
	
	\textit{(v)}
	We rewrite
	\begin{equation}\label{eq:asLast}
		\frac{\Phi_{n+1,N}^{(\beta)}(z)}{\Phi_{n,N}^{(\beta)}(z)}=
		\frac{\Phi_{n+1,N}^{(\beta)}(z)}{\Phi_{n,N}(z)}
		\frac{\Phi_{n-1,N}(z)}{\Phi_{n,N}^{(\beta)}(z)}
		\frac{\Phi_{n,N}(z)}{\Phi_{n-1,N}(z)}.
	\end{equation}
	
	Suppose first that ${\alpha}\ne 0$. As discussed above, $m_{-\bar{\alpha}\beta}(z)$ is finite and nonzero on $\C\setminus{\Lambda_{\beta,t}}$. Therefore by \textit{(i)} and \textit{(iv)}, the three factors on the right-hand side of~\eqref{eq:asLast} converge uniformly on compacts of $z\in\C\setminus{\Lambda_{\beta,t}}$ to $1/m_{-\bar{\alpha}\beta}(z)$, $m_{-\bar{\alpha} \beta}(z)$, and $G_a(z)$, respectively. 
	
	If {$\alpha=0$} then $m_0(z) = 0$ for $z\in\D$ by Lemma~\ref{lem:m_and_G}{\it (i)}, so the above argument does not apply in $\D$. However it still works for $z\in\C\setminus\overline\D$, where $m_0(z)$ is nonzero, yielding uniform convergence of~\eqref{eq:asLast} on compacts of $\C\setminus\overline\D$ to $G_0(z) \equiv 1$. To get uniform convergence on compacts of $\D$ we use~\eqref{eq:POPUCSymmetry}
	to get that 
	\[
	\frac{\Phi_{n+1,N}^{(\beta)}(z)}{\Phi_{n,N}^{(\beta)}(z)}\to z\overline{G_0(\bar{z}^{-1})} = z. 
	\]
	This proves {\it(v)} on $\C\setminus\T$. By Remark~\ref{rem:spec_incl_lam},  $\Lambda_{\beta,t}\equiv\T$ when $\alpha=0$, so {\it(v)} holds on compacts of $\C\setminus\Lambda_{\beta,t}$. The proof is complete. 
\end{proof}

\begin{rem}
	Analogous asymptotics can also be obtained for the case when instead of~\eqref{eq:convergence} we require (the variable analogue of) the more general L\'{o}pez condition 
\[
		 \lim_{n/N\to t}|\alpha_{n,N}|=a
	\quad\mbox{ and }\quad 
		 \lim_{n/N\to t}\frac{\alpha_{n+1,N}}{\alpha_{n,N}}=\lambda,
\]
with $a> 0$ and $\lambda\in\T$.  In this setting, one replaces $\beta$ by $\beta_{n,N}=\gamma \alpha_{n-1,N}/|\alpha_{n-1,N}|$, for any $\gamma\in\T$, and the same argument goes through. The limit functions $m_{-\bar\alpha\beta}$ and $f_{-\bar\alpha\beta}$ are then  replaced by the $m$-function and the Schur function of the CMV matrix $\mathcal{C}(- a\gamma, -\bar\lambda a\gamma,-\bar\lambda^2 a\gamma,\ldots)$. 
\end{rem}

\subsection{Asymptotically periodic Verblunsky coefficients}\label{subsec:RatioPeriodic}

The very same proof extends to more general situations. In particular, consider the asymptotically periodic setting, where
\begin{equation}\label{eq:almostPeriodic}
	\lim_{kp/N\to t}\alpha_{kp+j,N}=\alpha_{j}, \quad \mbox{for all }  j\in\{0,\dots,p-1\},
\end{equation}
for some $(\alpha_j)_{j=0}^{p-1}\in\D^p\setminus\{(0,\ldots,0)\}$ and $t>0$. We will treat $\alpha_j$ as a $p$-periodic sequence in what follows.

Note that here we assume that $\alpha_j\in\D$, that is, we exclude the case when some of $\alpha_j$'s are unimodular. This is done to simplify the exposition. With enough care this case can be easily accommodated using the same methods.

Denote by $\mathcal{C}_{\beta}$ the $p$-periodic CMV matrix with the first $p$ Verblunsky coefficients $-\bar\alpha_{p-1} \beta,-\bar\alpha_{p-2} \beta,\ldots,-\bar\alpha_{0} \beta$. Let $m_{\beta}$ and $f_\beta$ be the $m$-function and the Schur function of $\mathcal{C}_{\beta}$. For $j\in\N_{0}$, let $m^{(j)}_{\beta}$ and $f^{(j)}_\beta$ be the $m$-function and the Schur function of $\mathcal{C}^{(j)}_{\beta}$, the $j$ times stripped CMV matrix $\mathcal{C}_{\beta}$ (see Section~\ref{subsec:CMV}). 
Let $\Delta$ be the discriminant of  $\mathcal{C}_{\beta}$ (which is independent of $\beta$ and coincides with the discriminant of each $\mathcal{C}^{(j)}_{\beta}$). Recall the sets $\Lambda_{\beta,t}$ and $\Lambda_{\beta}$ from~\eqref{eq:Lambda1} and~\eqref{eq:Lambda2}.
We also introduce 
\begin{align*}
	\calZ(f)&= \{z\in\D \mid f(z)=0 \},
	\\
	\calP(f)&= \{z\in\C\setminus\overline{\D} \mid f(z)= \infty \},
\end{align*}
for the sets of zeros and poles of a Schur function $f$. By analyticity of $f$ on $\D$ and~\eqref{eq:symSchur}, the set of zeros and poles of $f$ 
are indeed subsets of $\D$ and $\C\setminus\overline{\D}$, respectively. 

\begin{thm}\label{thm:ratioPeriodic}
	Let $\beta\in\T$, $t>0$, and let $\{\alpha_{n,N} \mid n,N\in\N_{0}\}\subset\D$ satisfy~\eqref{eq:almostPeriodic}  with $(\alpha_j)_{j=0}^{p-1}\in\D^p\setminus\{(0,\ldots,0)\}$. Fix any $j\in \N_0$.
	Then the following limits hold:
	\begin{enumerate}[(i)]
		\item
		\[
		\lim_{kp/N\to t}\frac{\Phi_{kp-j,N}(z)}{\Phi_{kp-j+1,N}^{(\beta)}(z)}= m^{(j)}_{\beta}(z) 
		\]
		uniformly on compacts of $\C\setminus{\Lambda_{\beta,t}}$.
		\item
		\[
		\lim_{kp/N\to t}\frac{\Phi_{kp-j,N}(z)}{\Phi_{kp-j,N}^{*}(z)}= f^{(j)}_{1}(z)
		\]
		uniformly on compacts of $\C\setminus({\Lambda_t}\cup \calP(f^{(j)}_{1}))$.
		\item
		\[
		\lim_{kp/N\to t}\dfrac{\Phi^*_{kp-j+1,N}(z)}{\Phi^*_{kp-j,N}(z)}=1-z \alpha_{p-j} f_1^{(j)}(z)
		\]
		uniformly on compacts of $\C\setminus({\Lambda_t}\cup \calP(f^{(j)}_{1}))$.
		\item
		\[
		\lim_{kp/N\to t}\dfrac{\Phi_{kp-j+1,N}(z)}{\Phi_{kp-j,N}(z)}=z-\bar\alpha_{p-j}{/f_1^{(j)}(z)}
		\]
		uniformly on compacts of $\C\setminus({\Lambda_t}\cup \calZ(f^{(j)}_{1}))$.
		\item
		\begin{align}
			\label{eq:ratio_popuc_Gper1}
			\lim_{kp/N\to t}\frac{\Phi_{kp-j+1,N}^{(\beta)}(z)}{\Phi_{kp-j,N}^{(\beta)}(z)}
			&=
			\frac{m_\beta^{(j-1)}(z)}{m_\beta^{(j)}(z)} (z-\bar\alpha_{p-j+1}/{f_1^{(j-1)}(z)})
			\\
			\label{eq:ratio_popuc_Gper2}
			& = \frac{\overline{m_\beta^{(j-1)}(\bar{z}^{-1})}}{\overline{m_\beta^{(j)}(\bar{z}^{-1})}} (1-z \alpha_{p-j+1} f_1^{(j-1)}(z))
		\end{align}
		uniformly on compacts of $ \C\setminus \Lambda_{\beta,t}$.
	\end{enumerate}
\end{thm}
\begin{rem}
We remark that the same argument applies to the setting of right limits, using essentially the same proof (in particular, one can allow $\alpha_{k}\in\overline{\D}$). We do not pursue this here. See~\cite{BRZ} for a related setting. 
\end{rem}
\begin{proof}
	The proof of {\it(i)} is identical to that of Theorem~\ref{thm:ratio}{\it(i)}. 
	
	The only difference in the proof of {\it(ii)} is that, unlike in Theorem~\ref{thm:ratio}{\it(ii)}, where $m_{-\bar\alpha\beta}(\bar{z}^{-1})$ never vanishes on $\C\setminus\Lambda_{\beta,t}$, in our case the $m$-function may vanish. We therefore need to avoid points $z$ where $m_{\beta}^{(j)}(\bar{z}^{-1}) = 0$. By Lemma~\ref{lem:m}{\it{(ii)}} and~\eqref{eq:symSchur}, these points are exactly those in $\calP(f_{\beta}^{(j)})$. This set is $\beta$-independent, so we may take $\beta=1$, and the proof of {\it(ii)} is complete. 
	
	The only difference in the proof of {\it (iii)}, compared to Theorem~\ref{thm:ratio}, is the absence of a direct connection to the function $G_\Delta$ which is why the final expression for the limit in {\it (iii)} have a different form. When deriving {\it (iv)} from {\it (iii)}, we apply~\eqref{eq:symSchur} (instead of~\eqref{eq:GSymm}).
	
The result in~\eqref{eq:ratio_popuc_Gper1} follows from \eqref{eq:asLast} (with $n=kp-j$) together with {\it(i)} and {\it{(iv)}} uniformly on compacts of $\C\setminus(\Lambda_{\beta,t}\cup \calZ(f^{(j-1)}_{1}))$. Similarly, \eqref{eq:ratio_popuc_Gper2} follows by instead writing
	\begin{equation*}
		\frac{\Phi_{n+1,N}^{(\beta)}(z)}{\Phi_{n,N}^{(\beta)}(z)}=
		\frac{\Phi_{n+1,N}^{(\beta)}(z)}{\Phi^*_{n,N}(z)}
		\frac{\Phi^*_{n-1,N}(z)}{\Phi_{n,N}^{(\beta)}(z)}
		\frac{\Phi^*_{n,N}(z)}{\Phi^*_{n-1,N}(z)},
	\end{equation*}
	and arguing as in~\eqref{eq:starPOPUCAs} while applying {\it(i)} and {\it{(iii)}}, uniformly on compacts of $\C\setminus(\Lambda_{\beta,t}\cup \calP(f^{(j-1)}_{1}))$. Since $\calZ(f^{(j-1)}_{1})$ belongs to $\D$ and $ \calP(f^{(j-1)}_{1})$ to $\C\setminus\overline{\D}$, we obtain uniform convergence on compacts of $ \C\setminus \Lambda_{\beta,t}$, as claimed.
\end{proof}

Barrios and L\'{o}pez~\cite[Thm.~2]{BarLop}, using a very different proof, computed the limit of the ratios $\Phi_{n+p}(z)/\Phi_n(z)$. This result can be recovered easily as a corollary of the previous theorem. 

\begin{cor}\label{cor:BarLop}
	Let $\beta\in\T$, $t>0$, and let $\{\alpha_{n,N} \mid n,N\in\N_{0}\}\subset\D$ satisfy~\eqref{eq:almostPeriodic}  with $(\alpha_j)_{j=0}^{p-1}\in\D^p\setminus\{(0,\ldots,0)\}$. 
	Then 
	\begin{equation}
		\lim_{n/N\to t}\frac{\Phi_{n+p,N}^{(\beta)}(z)}{\Phi_{n,N}^{(\beta)}(z)} =
		\lim_{n/N\to t}\frac{\Phi_{n+p,N}(z)}{\Phi_{n,N}(z)}  =
		\lim_{n/N\to t}\frac{\Phi^*_{n+p,N}(z)}{\Phi^*_{n,N}(z)}= G_\Delta(z),
		\label{eq:ratio_asympt_var_periodic}
	\end{equation}
	and the convergence is uniform on compacts of $\C\setminus\Lambda_{\beta,t}$, $\C\setminus(\Lambda_{t}\cup \calZ)$, and $\C\setminus(\Lambda_{t}\cup\calP)$, respectively, where $\calZ\subset\D$ and $\calP\subset\C\setminus\overline{\D}$ are given by 
\[
		\calZ= \bigcup_{j=0}^{p-1} \calZ(f^{(j)}_{1}) 
		\quad\mbox{ and }\quad
		\calP= \bigcup_{j=0}^{p-1} \calP(f^{(j)}_{1}).
\]
\end{cor}
\begin{proof}
	Taking the product of $p$ consecutive ratios with $j=1,\ldots,p$ in  {\it (iii)}, {\it (iv)}, {\it (v)}, and using periodicity $\mathcal{C}_\beta^{(p)} = \mathcal{C}_\beta$, we obtain that the three limits exist and coincide with the value
	\begin{equation}\label{eq:ratioCommon}
		L(z)=\prod_{j=1}^{p} (1-z \alpha_{p-j+1} f_1^{(j-1)}(z))
		= \prod_{j=1}^{p}(z-\bar\alpha_{p-j+1}{/f_1^{(j-1)}(z)} ),
	\end{equation}
	uniformly on compacts of the corresponding sets. 
	
	Notice that $L(z)$ is fully determined by the sequence $(\alpha_j)_{j=0}^{p-1}$. So without loss of generality we can assume that $\alpha_{kp+j,N}=\alpha_{j}$ for all $k,N\in\N_{0}$ and $j=0,\ldots,p-1$. It is an elementary calculus exercise to show that if a complex sequence $\{x_k\}$ satisfies $x_{k+1}/x_k\to a$ and $|x_k|^{1/k}\to b$, as $k\to\infty$, then $|a|=b$. Applying this to $x_k = \Phi_{kp}(z)$ and using~\eqref{eq:root} with~\eqref{eq:ratioCommon}, we obtain 
	\begin{equation}
		\label{eq:LandG}
		|L(z)|= [\exp(- U_{\nu_\Delta}(z))]^p = |G_\Delta(z)|,
	\end{equation}
	see~\eqref{eq:U_eq_log_G_Delta}.
	
	Now note that $L(z)$ in~\eqref{eq:ratioCommon} is analytic nonvanishing function on $\D$, bearing in mind that $|f_1^{(j-1)}(z)|\le 1$ for Schur functions. Function $G_\Delta$ is analytic on $\D$ and it is also nonvanishing, by~\eqref{eq:LandG}. Finally, $L(0)=G_{\Delta}(0)=1$ by~\eqref{eq:ratioCommon} and~\eqref{eq:ptyG1}. Then applying the maximum modulus principle to $L/G_\Delta$ shows that $L=G_\Delta$ and completes the proof.
\end{proof}

\section{Asymptotic zero density for variable POPUC}\label{sec:var_popuc}

Recall that $\nu_{n,N}^{(\beta)}$ denotes the zero-counting measure of $\Phi_{n,N}^{(\beta)}$, see~\eqref{eq:count-meas_var}. The following theorem identifies its weak limit under the natural condition~\eqref{eq:VerblunskyStrong}. We provide two different proofs.

\begin{thm}\label{thm:popuc_asympt_distr}
	Let $t>0$ and $\alpha:[0,t]\to\overline{\D}$ be Lebesgue  measurable. Suppose further that $\{\alpha_{n,N} \mid n,N\in\N_{0} \}\subset\D$ is such that
	\begin{equation}\label{eq:varyingNevai}
		\lim_{n/N\to s} \alpha_{n,N}=\alpha(s), \,\mbox{ for {a.e.} } s\in[0,t].
	\end{equation}
	Then, for any $\beta\in\T$,
	\begin{equation}\label{eq:weak_limit}
		\wlim_{n/N\to t} \nu_{n,N}^{(\beta)}=\frac{1}{t}\int_{0}^{t}\nu_{|\alpha(s)|}\dd s,
	\end{equation}
	where the measure $\nu_{|\alpha(s)|}$ is given by~\eqref{eq:meas_nu} if $|\alpha(s)|\neq1$, and equals $\delta_{-1}$, otherwise.
\end{thm}

\begin{rem}\label{rem:var_dens_explicit}
	Assuming $\alpha$ is a continuous function on $[0,t]$, then the limiting measure
	\begin{equation}
		\sigma_{t}=\frac{1}{t}\int_{0}^{t}\nu_{|\alpha(s)|}\dd s
		\label{eq:def_sigma_t}
	\end{equation}
	in~\eqref{eq:weak_limit} is an average of the equilibrium measures $\nu_{|\alpha(s)|}$ over $s\in[0,t]$ supported on the arc $\{\ee^{\ii\theta} \mid \theta_{|\alpha^{*}(t)|}\leq \theta\leq 2\pi-\theta_{|\alpha^{*}(t)|}\}$, where
	\[
	|\alpha^{*}(t)|=\min_{s\in[0,t]}|\alpha(s)|
	\]
	{and $\theta_{|\alpha^{*}(t)|}$ given by~\eqref{eq:def_theta}.}
	If, in addition, values of $\alpha$ stays away from~$\T$, i.e., $|\alpha(s)|<1$ for all $s\in[0,t]$, then $\sigma_{t}$ is absolutely continuous with the density
	\begin{equation}
		\frac{\dd\sigma_{t}}{\dd\theta}\left(\ee^{\ii\theta}\right)=\frac{1}{2\pi t}\int_{0}^{t}\frac{\sin\left(\theta/2\right)}{\sqrt{\sin^{2}\left(\theta/2\right)-|\alpha(s)|^{2}}}\,\chi_{\left[\theta_{|\alpha(s)|},2\pi-\theta_{|\alpha(s)|}\right]}(\theta)\dd s,
		\label{eq:density_var_explic}
	\end{equation}
	where $\chi_{[a,b]}$ is the indicator function of an interval $[a,b]$.
\end{rem}


\subsection{First proof based on ratio asymptotics and potential-theoretic argument}\label{sec:FirstProof}

The first proof of Theorem~\ref{thm:popuc_asympt_distr} is based on a determination of the limiting logarithmic potential, which is the approach that was successfully applied in the case of orthogonal polynomials on the real line in~\cite{kuijlaarsvanassche_jat99}. It relies on a general argument from potential theory which is sometimes referred to as Widom's lemma, see~\cite{widom_otaa90,widom_otaa94}. Suppose a sequence of measures $\mu_{n}$ uniformly supported in a compact subset of~$\C$ is given. Denote their logarithmic potentials as
\[
U_{\mu_{n}}(z)={-}\int_{\C}\log|z-\xi|\dd\mu_{n}(\xi), \quad z\in\C\setminus\supp\mu_{n}.
\]
By Widom's lemma, if $U_{\mu_{n}}(z)$ tends to a limiting function $U(z)$, as $n\to\infty$, for a.e. $z\in\C$ (with respect to the two-dimensional Lebesgue measure in $\C$), then $\mu_{n}$ converges weakly to a~measure $\mu$, as $n\to\infty$, and $U$ is the logarithmic potential of~$\mu$. The limiting measure $\mu$ can be determined from $U$ by the formula $2\pi\mu=\Delta U$ which is to be understood in the distributional sense.

The following auxiliary result provides a lower and upper bound for the ratio of two consecutive POPUC which will be needed later.
\begin{lem}\label{lem:ratio_bounds}
	Let $\beta\in\T$ and $0<r\leq R$. Then, for all $z\in\C$ in the annulus $r\leq|z|\leq R$ and all $n\in\N$, one has
	\begin{equation}
		\frac{(r-1)^2}{R+1} < \left|\frac{\Phi_{n+1}^{(\beta)}(z)}{\Phi_{n}^{(\beta)}(z)}\right| < \frac{(R+1)^2}{r-1}, \quad \mbox{ if }\; 1<r\le |z|\le R,
		\label{eq:ratio_unif_bound}
	\end{equation}
	and
	\begin{equation}
		\frac{r^{2}(1-R)^2}{R^{2}(1+r)} < \left|\frac{\Phi_{n+1}^{(\beta)}(z)}{\Phi_{n}^{(\beta)}(z)}\right| < \frac{R^{2}(1+r)^2}{r^{2}(1-R)}, \quad \mbox{ if }\; r\le|z|\le R<1.
		\label{eq:ratio_unif_bound_inside_D}
	\end{equation}
\end{lem}

\begin{proof}
	Notice that ~\eqref{eq:ratio_unif_bound} together with~\eqref{eq:POPUCSymmetry} imply~\eqref{eq:ratio_unif_bound_inside_D}. Therefore it suffices to prove~\eqref{eq:ratio_unif_bound}.
	
	Fix $\beta\in\T$ and $1<r\leq R$. For any $z\in\C\setminus\overline{\D}$, we may write
	\begin{equation}\label{eq:ratio_bound0}
		\left|\frac{\Phi_{n+1}^{(\beta)}(z)}{\Phi_{n}^{(\beta)}(z)}\right|=
		\left|\frac{\Phi_{n+1}^{(\beta)}(z)}{\Phi_{n}(z)}\right|\times\left|\frac{\Phi_{n}(z)}{\Phi_{n}^{(\beta)}(z)}\right|.
	\end{equation}
	Since $\Phi_n/\Phi_n^*$ is a Blaschke product {one infers, for all $\C\setminus\overline{\D}$, that}
	\begin{equation}\label{eq:ratio_bound1}
		\left| \frac{\Phi_n^*(z)}{\Phi_n(z)} \right|<1.
	\end{equation}
	By the Szeg\H{o} recursion~\eqref{eq:POPUC}, we have
	\[
	\frac{\Phi_{n+1}^{(\beta)}(z)}{\Phi_n(z)} =  z - \bar{\beta}\frac{\Phi_n^*(z)}{\Phi_n(z)}.
	\]
	Hence, by making use of~\eqref{eq:ratio_bound1}, we deduce that
	\begin{equation}\label{eq:ratio_bound2}
		r-1 < \left| \frac{\Phi_{n+1}^{(\beta)}(z)}{\Phi_n(z)} \right|  < R+1, \quad \mbox{ for } r\leq|z|\leq R.
	\end{equation}
	Next, by the Szeg\H{o} recursion~\eqref{eq:recur_opuc} and ~\eqref{eq:POPUC} again,
	\[
		\frac{\Phi_{n}(z)}{\Phi_n^{(\beta)}(z)} = \frac{z - \bar{\alpha}_{n-1}\Phi_{n-1}^{*}(z)\big/\Phi_{n-1}(z)}{z -\bar{\beta}\Phi_{n-1}^{*}(z)\big/\Phi_{n-1}(z)},
	\]
	and hence, using~\eqref{eq:ratio_bound1}, we get
	\begin{equation}\label{eq:ratio_bound4}
		\frac{r-1}{R+1} < \left| \frac{\Phi_{n}(z)}{\Phi_n^{(\beta)}(z)} \right| <\frac{R+1}{r-1}, \quad \mbox{ for } r\leq|z|\leq R.
	\end{equation}
	Finally, equation~\eqref{eq:ratio_bound0} together with bounds~\eqref{eq:ratio_bound2} and~\eqref{eq:ratio_bound4} implies~\eqref{eq:ratio_unif_bound}.
\end{proof}

\begin{proof}[First proof of Theorem~\ref{thm:popuc_asympt_distr}]
	For $z\notin\T$, the logarithmic potential of $\nu_{n,N}^{(\beta)}$ can be written as
	\[
	U_{\nu_{n,N}^{(\beta)}}(z)=-\frac{1}{n}\log|\Phi_{n,N}^{(\beta)}(z)|=-\frac{1}{n}\sum_{k=0}^{n-1}\log\left|\frac{\Phi_{k+1,N}^{(\beta)}(z)}{\Phi_{k,N}^{(\beta)}(z)}\right|\!.
	\]
	Note that the above sum is the Riemann sum for a piece-wise constant function, i.e.
	\begin{equation}
		U_{\nu_{n,N}^{(\beta)}}(z)=-\int_{0}^{1}\log\left|\frac{\Phi_{[ns]+1,N}^{(\beta)}(z)}{\Phi_{[ns],N}^{(\beta)}(z)}\right|\dd s,
		\label{eq:log_pot_as_int}
	\end{equation}
	for $z\notin\T$, where $[x]$ is the integer part of $x\in\R$.
	
	Since $[ns]/N\to st$, as $n/N\to t$, {and taking assumption~\eqref{eq:varyingNevai} into account}, we may apply formula~\eqref{eq:ratio_popuc_G} getting
	\[
	\lim_{n/N\to t}\frac{\Phi_{[ns]+1,N}^{(\beta)}(z)}{\Phi_{[ns],N}^{(\beta)}(z)}=G_{|\alpha(st)|}(z),
	\]
	for all $z\in\C\setminus\T$ and {a.e.} $s\in[0,1]$. Moreover, for $z$ in compacts of $\C\setminus\T$, the argument of the
	logarithm in~\eqref{eq:log_pot_as_int} stays uniformly bounded away from $0$ and $\infty$, as it follows
	from~\eqref{eq:ratio_unif_bound} and~\eqref{eq:ratio_unif_bound_inside_D}.
	Thus, we may apply the Lebesgue Dominated Convergence Theorem in~\eqref{eq:log_pot_as_int} getting
	\begin{equation}   \label{eq:lim_log_pot_inproof}
		\lim_{n/N\to t}U_{\nu_{n,N}^{(\beta)}}(z)
		=-\int_{0}^{1}\log\left|G_{|\alpha(st)|}(z)\right|\dd s
		=-\frac{1}{t}\int_{0}^{t}\log\left|G_{|\alpha(s)|}(z)\right|\dd s,
	\end{equation}
	for all $z\in\C\setminus\T$. 
	Recalling~\eqref{eq:U_eq_log_G}, one can see that  the right-hand side of~\eqref{eq:lim_log_pot_inproof} is equal to  
	\[
	U_{\sigma_{t}}(z)=\frac{1}{t}\int_{0}^{t}U_{\nu_{|\alpha(s)|}}(z)\dd s, \quad z\in\C\setminus\T,
	\]
	the logarithmic potential of the average measure $\sigma_{t}$ defined by~\eqref{eq:def_sigma_t}.
	Now the statement follows from Widom's lemma.
\end{proof}

\subsection{Second proof based on method of moments}\label{sec:SecondProof}

We provide an alternative proof of Theorem~\ref{thm:popuc_asympt_distr} based on the method of moments. This approach goes back to the original paper by {Kac, Murdock, and Szeg\H{o}}~\cite{kms_jrma53}, see also {Bourget, Loya, and McMillan}~\cite{bourget-etal_jjm18}. 
Its advantage, apart from its simplicity, is that it adapts naturally to a variety of other contexts, such as  asymptotically periodic, block, or block asymptotically periodic POPUC and real line settings.

Let us introduce a notation first.
For each $N\in\N_0$, consider a family of two-sided \emph{block} tridiagonal matrices
\begin{equation}\label{eq:tridiagonal}
	J^{(N)}=\begin{pmatrix}
		\ddots & \ddots      & \ddots      & \ddots           & \ddots           & \ddots           &        \\
		\ddots & a^{(N)}_{-2,-2}   & a^{(N)}_{-2,-1}   & 0           & 0           & \ddots           &  \ddots      \\
		\ddots & a^{(N)}_{-1,-2}   & a^{(N)}_{-1,-1}   & a^{(N)}_{-1,0}    & 0           & 0           &   \ddots     \\
		\ddots      & 0           & a^{(N)}_{0,-1}    & a^{(N)}_{0,0}     & a^{(N)}_{0,1}     & 0           & \ddots      \\
		\ddots      & 0           & 0           & a^{(N)}_{1,0}     & a^{(N)}_{1,1}     & a^{(N)}_{1,2}     & \ddots \\
		      & \ddots           & 0           & 0           & \ddots      & \ddots      & \ddots 
	\end{pmatrix},
\end{equation}
where each $a_{j,k}^{(N)}$ ($j,k\in\Z$, $|j-k|\le 1$) is an {$\ell\times\ell$} complex matrix. 
Any finite-band matrix can be written in this form, including CMV matrices; the only distinction is that $J^{(N)}$ is doubly infinite. This choice is merely for convenience: it leads to a simpler expression for the entries of integer powers of $J^{(N)}$, which will play a role in the proof of Theorem~\ref{thm:popuc_asympt_distrGen} below.

We assume that the collection of matrices $\{a_{j,k}^{(N)} \mid j,k\in\Z, |j-k|\le 1, N\in\N_{0}\}$ is uniformly bounded. Then $J^{(N)}$, regarded as an operator on $\ell^{2}(\Z)$, is bounded.

Let $P_n$ be the orthogonal projection in {$\ell^{2}(\Z)$} onto the span of {$\{e_j \mid 0\le j \le n-1\}$}. Denote $J_n^{(N)}=P_n J^{(N)} P_n {\upharpoonright \Ran P_n}$,
i.e., $J_{n}^(N)$ is the $n\times n$ submatrix of $J^{(N)}$ corresponding to the row and column indices $j=0,\dots,n-1$.
Denote further by $\nu_{n,N}$ the eigenvalue counting measure of $J_{n}^{(N)}$, i.e.
\[
\nu_{n,N}=\frac{1}{n}\sum_{j=1}^{n}\delta_{z_{j}},
\]
where $z_{1},\dots,z_{n}$ are the eigenvalues of $J_n^{(N)}$ counted with multiplicity. 

	For three fixed complex $\ell\times\ell$ matrices $\gamma_{-1},\gamma_{0},\gamma_{1}$, let $J(\gamma_{-1},\gamma_{0},\gamma_{1})$ be the tridiagonal block (Laurent) matrix of the form~\eqref{eq:tridiagonal} with $a_{j-1,j}^{(N)}=\gamma_{-1}$, $a_{j,j}^{(N)}=\gamma_{0}$, $a_{j+1,j}^{(N)}=\gamma_{1}$ for all $j\in\Z$. Let $\nu_{n}^{(\gamma_{-1},\gamma_{0},\gamma_{1})}$ be the eigenvalue counting measure of the $n\times n$ truncation $J_{n}^{(N)}$ of $J(\gamma_{-1},\gamma_{0},\gamma_{1})$ as above. The existence of the weak limit of $\nu_{n}^{(\gamma_{-1},\gamma_{0},\gamma_{1})}$ as $n\to\infty$ need not be guaranteed in general. In the next statement, we assume a~weaker condition that the nonnegative integer moments of $\nu_{n}^{(\gamma_{-1},\gamma_{0},\gamma_{1})}$ converge to the corresponding moments of a~probability measure on $\C$.

\begin{thm}\label{thm:popuc_asympt_distrGen}
	Let {$\ell\in\N$, $k\in\N_{0}$,} $t>0$, and let $\gamma_{-1},\gamma_0,\gamma_1:[0,t]\to\C$ be three  measurable $\ell\times\ell$ matrix-valued functions. Suppose that 
	$a_{j,k}^{(N)}$ is a collection of uniformly bounded matrices such that, for each $i\in\{-1,0,1\}$, 
		\begin{equation}\label{eq:limitsAssumptions}
		\lim_{n/N\to s} a_{n+i,n}^{(N)}=\gamma_{i}(s), \quad \mbox{for a.e. } s\in[0,t].
		\end{equation}
		Moreover, assume that, for a.e. $s\in[0,t]$, there exists a measure 
		$\nu^{(s)}$, for which 
\begin{equation}
	\label{eq:lim_mom_const_matr}
	\lim_{n\to\infty}\int z^{k} d\nu_{n}^{(\gamma_{-1}(s),\gamma_0(s),\gamma_1(s))} = \int z^{k} d\nu^{(s)}, \quad k\in\N_0.
\end{equation}
		Then
		\begin{equation}\label{eq:weak_limitGen}
			\wlim_{n/N\to t} \int z^k \dd\nu_{n,N}=\frac{1}{t}\int_{0}^{t} \int z^k \dd\nu^{(s)}\,\dd s, \quad k\in\N_0.
		\end{equation}
\end{thm}

\begin{proof}
For $k=0$ the claim is trivial. Suppose $k\in\N$. First note that 
	\begin{equation}\label{eq:traces1}
		\wlim_{n/N\to t} \int z^k \dd\nu_{n,N}=\lim_{n/N\to t}\frac{1}{n} \Tr\left(P_n J^{(N)} P_n\right)^{k}=	\lim_{n/N\to t}\frac{1}{n} \Tr P_n \left(J^{(N)}\right)^{k}  P_n,
	\end{equation}
where the second equality is due the banded structure of $J^{(N)}$. Next we denote
	\[
	\mathcal{D}_{k}=\left\{(d_1,\dots,d_k)\in\{-1,0,1\}^{k} \;\bigg|\; \sum_{s=1}^{k}d_{s}=0 \right\}.
	\]
	Clearly, the set $\mathcal{D}_{k}$ is finite.
	Using the definition of matrix multiplication, we can rewrite the expression on the right-hand side of~\eqref{eq:traces1} as
	\[
	\frac{1}{n}\sum_{j=0}^{n-1} \Tr \sum_{(d_1,\dots,d_{k})\in \mathcal{D}_{k}} a_{j,j+d_1}^{(N)} a_{j+d_1,j+d_1+d_2}^{(N)} \dots a_{j+d_1+\dots+d_{k-1},j}^{(N)}.
	\]
	This expression can be represented as the integral of a piecewise constant function, namely, it is equal to
\[
		\int_0^1 \Tr \sum_{(d_1,\dots,d_{k})\in \mathcal{D}_{k}} a_{[ns],[ns]+d_1}^{(N)} a_{[ns]+d_1,[ns]+d_1+d_2}^{(N)} \dots a_{[ns]+d_1+\dots+d_{k-1},[ns]}^{(N)}\,\dd s.
\]
	Since we assume that the matrices $a_{j,k}^{(N)}$ are uniformly bounded and $[ns]/N\to st$ if $n/N\to t$, we may apply the Lebesgue dominated convergence and we get
	\begin{align}
		\wlim_{n/N\to t} \int z^k \dd\nu_{n,N}&=
		\int_0^1 \Tr \sum_{(d_1,\ldots,d_{k})\in \mathcal{D}_{k}} \gamma_{d_1}(st) \gamma_{d_2}(st) \ldots \gamma_{d_{k}}(st) \,\dd s \nonumber \\
		&=
		\frac{1}{t} \int_0^t \Tr \sum_{(d_1,\ldots,d_{k})\in \mathcal{D}_{k}} \gamma_{d_1}(s) \gamma_{d_2}(s) \ldots \gamma_{d_{k}}(s) \,\dd s.
		\label{eq:traces3}
	\end{align}

	Now, fix $s=s_{0}\in[0,t]$ for which the assumption~\eqref{eq:lim_mom_const_matr} holds and apply the same procedure in the particular case when $J^{(N)}\equiv J(\gamma_{-1}(s_0),\gamma_{0}(s_0),\gamma_{1}(s_0))$, i.e., when $a_{n-1,n}^{(N)} \equiv \gamma_{-1}(s_0)$, $a_{n,n}^{(N)} \equiv \gamma_{0}(s_0)$, $a_{n+1,n}^{(N)} \equiv \gamma_{1}(s_0)$  for all $n\in\Z$ and $N\in\N_{0}$. Then~\eqref{eq:traces3} yields the equality
		\begin{align*}
			\lim_{n\to\infty}\int z^{k} d\nu_{n}^{(\gamma_{-1}(s_0),\gamma_0(s_0),\gamma_1(s_0))}&=	
			\frac{1}{t} \int_0^t \Tr \sum_{(d_1,\ldots,d_{k})\in \mathcal{D}_{k}} \gamma_{d_1}(s_0) \gamma_{d_2}(s_0) \ldots \gamma_{d_{k}}(s_0) \,\dd s \\
			&=\Tr \sum_{(d_1,\ldots,d_{k})\in \mathcal{D}_{k}} \gamma_{d_1}(s_0) \gamma_{d_2}(s_0) \ldots \gamma_{d_{k}}(s_0).
		\end{align*}
		By~\eqref{eq:lim_mom_const_matr}, we get
		\[
		\int z^{k} d\nu^{(s_0)}=\Tr \sum_{(d_1,\ldots,d_{k})\in \mathcal{D}_{k}} \gamma_{d_1}(s_0) \gamma_{d_2}(s_0) \ldots \gamma_{d_{k}}(s_0),
		\]
		which holds true for a.e. $s_0\in[0,t]$. Substituting this expression into \eqref{eq:traces3} yields \eqref{eq:weak_limitGen}, completing the proof.
\end{proof}

\begin{proof}[Second proof of Theorem~\ref{thm:popuc_asympt_distr}]
	We apply Theorem~\ref{thm:popuc_asympt_distrGen} with $\ell=2$ and 
	$$
	J^{(N)}=\bm{0}\oplus\mathcal{C}^{(\beta)}_{n,N}\oplus\bm{0},
	$$
	where $\mathcal{C}_{n,N}^{(\beta)}=\mathcal{C}_n(\alpha_{0,N},\ldots,\alpha_{n-2,N},\beta)$ is our finite $n\times n$ unitary CMV matrix (see \eqref{eq:POPUCdef}--\eqref{eq:CMV_beta_matrix_abbrev}) acting on $\Ran P_n$; the first zero matrix acts on $\ell^{2}(-\N)$ and the second on $\ell^2(\{n,n+1,\ldots\})$. By the assumption~\eqref{eq:varyingNevai} and~\eqref{eq:CMV}, we obtain the existence, for a.e. $s\in[0,t]$, of limits~\eqref{eq:limitsAssumptions} with
	\begin{equation*}
		\gamma_{-1}(s)
		\!=\!\begin{pmatrix}
				0 & \overline{\alpha(s)}\rho(s) \\
				0 & \rho(s)^2
			\end{pmatrix}\!,
		\;
	\gamma_{0}(s)
	\!=\!\begin{pmatrix}
	-|\alpha(s)|^2 & \overline{\alpha(s)}\rho(s) \\
		- {\alpha(s)}\rho(s) & -|\alpha(s)|^2
	\end{pmatrix}\!,
	\;
	\gamma_{1}(s)
	\!=\!\begin{pmatrix}
		\rho(s)^2 & 0 \\
		- {\alpha(s)}\rho(s) & 0
	\end{pmatrix}\!,
	\end{equation*}
	where $\rho(s)=\sqrt{1-|\alpha(s)|^2}$. 
	
	We now need to check the convergence of moments~\eqref{eq:lim_mom_const_matr} of eigenvalue counting measure for the limiting matrix $J(\gamma_{-1}(s),\gamma_0(s),\gamma_1(s))$, which is the two-sided CMV matrix with constant coefficients $\alpha_n=\alpha(s)$, $n\in\Z$. 
	
To this end, introduce the $n\times n$ CMV matrix $\mathcal{C}_n^{(\beta,s)} = \mathcal{C}_n(\alpha(s),\ldots,\alpha(s),\beta)$. For each $n\in\N$, $P_nJ(\gamma_{-1}(s),\gamma_0(s),\gamma_1(s))P_n$ and $P_n\mathcal{C}_n^{(\beta,s)}  P_n$ differ only in 4 entries (2 in the top-left, and 2 at the bottom right), and therefore the nonegative integer moments of their eigenvalue counting measures are asymptotically equal. But the zero counting measure for $\mathcal{C}_n^{(\beta,s)}$ exists and is equal to $\nu_{|\alpha(s)|}$, see Section~\ref{sect:potential}. 
This establishes that~\eqref{eq:lim_mom_const_matr} holds with $\nu^{(s)}$ replaced by $\nu_{|\alpha(s)|}$. Theorem~\ref{thm:popuc_asympt_distrGen} then applies showing
%
		\begin{equation}
			\wlim_{n/N\to t}\int z^{k}\dd\nu_{n,N}^{(\beta)}=\frac{1}{t}\int_{0}^{t}\int z^{k}\dd\nu_{|\alpha(s)|} \dd s,
			\label{eq:wlim_mom_k_pos}
		\end{equation}
for all $k\in\N_{0}$.
		
Since measures $\nu_{n,N}^{(\beta)}$ as well as $\nu_{|\alpha(s)|}$ are probability measures supported in $\T$, equality~\eqref{eq:wlim_mom_k_pos} extends to all $k\in\Z$ by the complex conjugation. Finally, the weak convergence~\eqref{eq:weak_limit} follows by the Stone--Weierstrass theorem applied to the Laurent polynomials on the unit circle.
\end{proof}

\subsection{Asymptotically periodic variable Verblunsky coefficients}\label{subsec:distr_asympt_popuc_periodic}

Here we establish a generalization of Theorem~\ref{thm:popuc_asympt_distr} to asymptotically periodic setting. Both approaches presented earlier, the one based on the limiting logarithmic potential and the moment method, extend naturally to this framework.

\begin{thm}\label{thm:asympt_distr_popuc_periodic}
	Let $t>0$, $p\in\N$, and $\alpha_{0},\dots,\alpha_{p-1}:[0,t]\to\D$ be measurable functions. Suppose further that $\{\alpha_{n,N} \mid n,N\in\N_{0} \}\subset\D$ is such that, for all $j\in\{0,1,\dots,p-1\}$,
	\begin{equation}\label{eq:ratioAgain}
		\lim_{kp/N\to s} \alpha_{kp+j,N}=\alpha_{j}(s), \,\mbox{ for {a.e.} } s\in[0,t].
	\end{equation}
	Then, for any $\beta\in\T$, one has
	\begin{equation}\label{eq:weak_limit_popuc_periodic}
		\wlim_{n/N\to t} \nu_{n,N}^{(\beta)}=\frac{1}{t}\int_{0}^{t}\nu_{\Delta_{s}}\dd s,
	\end{equation}
	where the measure $\nu_{\Delta_{s}}$ is as in~\eqref{eq:def_nu_periodic} and~\eqref{eq:def_discr} with
	$\alpha_{j}$ replaced by $\alpha_{j}(s)$ for all $j\in\{0,1,\dots p-1\}$.
\end{thm}

\begin{proof}[First proof of Theorem~\ref{thm:asympt_distr_popuc_periodic}]
	Recall that under the assumption~\eqref{eq:ratioAgain}, POPUC satisfy the ratio asymptotics~\eqref{eq:ratio_asympt_var_periodic}.
	For $n=kp+j$ with $j\in\{0,1,\dots,p-1\}$ and $\beta\in\T$, one obviously has
	\[
	\Phi_{kp+j,N}^{(\beta)}(z)=\Phi_{j,N}^{(\beta)}(z)\prod_{i=0}^{k-1}\frac{\Phi_{(i+1)p+j,N}^{(\beta)}(z)}{\Phi_{ip+j,N}^{(\beta)}(z)},
	\]
	therefore the logarithmic potential of the zero-counting measure $\nu_{n,N}^{(\beta)}$ of $\Phi_{n,N}^{(\beta)}$ can be written as
	\begin{equation}
		U_{\nu_{n,N}^{(\beta)}}(z)=-\frac{1}{n}\log\left|\Phi_{j,N}^{(\beta)}(z)\right|-{\frac{k}{n}}\int_{0}^{1}\log\left|\frac{\Phi_{([ks]+1)p+j,N}^{(\beta)}(z)}{\Phi_{[ks]p+j,N}^{(\beta)}(z)}\right|\dd s.
		\label{eq:log_pot_as_int_periodic}
	\end{equation}
	By using the assumptions of Theorem~\ref{thm:asympt_distr_popuc_periodic}, limit relation~\eqref{eq:ratio_asympt_var_periodic}, and Lemma~\ref{lem:ratio_bounds}, we apply the Lebesgue dominated convergence in~\eqref{eq:log_pot_as_int_periodic} and find
	\[
	\lim_{n/N\to t}U_{\nu_{n,N}^{(\beta)}}(z)=-\frac{1}{p}\int_{0}^{1}\log\left|G_{\Delta_{st}}(z)\right|\dd s=-\frac{1}{pt}\int_{0}^{t}\log\left|G_{\Delta_{s}}(z)\right|\dd s.
	\]
	The last formula yields the convergence of the logarithmic potentials
	\[
	\lim_{n/N\to t}U_{\nu_{n,N}^{(\beta)}}(z)=U_{\sigma_{t}}(z),
	\]
	for all $z\in\C\setminus\T$, where measure $\sigma_{t}$ is the average measure
\[
		\sigma_{t}=\frac{1}{t}\int_{0}^{t}\nu_{\Delta_{s}}\dd s.
\]
	Widom's lemma now implies the the weak convergence of corresponding measures.
\end{proof}

\begin{proof}[Second proof of Theorem~\ref{thm:asympt_distr_popuc_periodic}]
	The proof proceeds analogously to the second proof of Theorem~\ref{thm:popuc_asympt_distr} using Theorem~\ref{thm:popuc_asympt_distrGen} with $\ell=2p$ and~\eqref{eq:nu_n_tends_to_nu_scalar_opuc}.
\end{proof}
\begin{rem}
	Similarly as in~\eqref{eq:meas_nu}, the density of measure~\eqref{eq:def_nu_periodic} can be expressed in a~more explicit form if $p=2$. Namely, one has
	\[
	\supp\nu_{\Delta}=\left\{\ee^{\ii\theta}\mid \theta\in[\theta^{(+)},\theta^{(-)}]\cup[-\theta^{(-)},-\theta^{(+)}]\right\},
	\]
	where $0\leq\theta^{(+)}<\theta^{(-)}\leq\pi$ are defined by
	\begin{equation}
		\theta^{(\pm)}=\arccos\left(\pm\sqrt{(1-|\alpha_{0}|^{2})(1-|\alpha_{1}|^{2})}-\Re\alpha_{0}\bar{\alpha}_{1}\right).
		\label{eq:def_theta_plus_minus}
	\end{equation}
	The density of~$\nu_{\Delta}$ reads
	\[
	\frac{\dd\nu_\Delta}{\dd\theta}\left(\ee^{\ii\theta}\right)=\frac{1}{2\pi}\frac{|\sin\theta|}{\sqrt{\left(\cos\theta^{(+)}-\cos\theta\right)\left(\cos\theta-\cos\theta^{(-)}\right)}},
	\]
	for $\pm\theta\in[\theta^{(+)},\theta^{(-)}]$; see~\cite[Ex.~11.1.5]{simon_opuc2}.
	Then similarly as in Remark~\ref{rem:var_dens_explicit}, we can deduce the 2-periodic generalization of the formula for the density~\eqref{eq:density_var_explic} {of the average measure~\eqref{eq:weak_limit_popuc_periodic} at the point $\ee^{\ii\theta}$ which reads}
	\[
	\frac{1}{2\pi t}\int_{0}^{t}\frac{|\sin\theta|}{\sqrt{\left(\cos\theta^{(+)}_{s}-\cos\theta\right)\left(\cos\theta-\cos\theta^{(-)}_{s}\right)}}\chi_{\left[\theta_{s}^{(+)},\theta_{s}^{(-)}\right]\cup\left[-\theta_{s}^{(-)},-\theta_{s}^{(+)}\right]}(\theta)\,\dd s,
	\]
	for $\theta\in(-\pi,\pi)$, where $\theta_{s}^{(\pm)}$ is defined by~\eqref{eq:def_theta_plus_minus} with $\alpha_{0}$ and $\alpha_{1}$ replaced by $\alpha_{0}(s)$ and $\alpha_{1}(s)$.
\end{rem}

\section{Asymptotic zero density for variable OPUC}\label{sec:var_opuc}

In this section, we extend the ideas of Theorem~\ref{thm:popuc_asympt_distr} to the setting of zeros of OPUC.
Recall from Section~\ref{sec:root_asym} that, even in the non-varying case, the behavior of OPUC zeros is more intricate than that of POPUC. In particular, when $\alpha_n \to 0$, the rate of convergence becomes important. {More concretely, we prove a variable generalization of the fact that the asymptotic distribution of zeros of OPUC is given by the normalized Lebesgue measure on $\T$ if $\alpha_{n}\to0$ with $|\alpha_{n}|^{1/n} \to 1$. On the other hand, if the latter condition is altered to    $|\alpha_n|^{1/n} \to A\in[0,1)$ as in the Mhaskar--Saff theorem, we do not have a variable generalization. This remains an open problem (see Remark~\ref{rem:open_prob} below).}

To establish a partial OPUC analogue of Theorem~\ref{thm:popuc_asympt_distr}, we follow the operator-theoretic approach of Simon~\cite[Sec.~8.2]{simon_opuc1}, which employs CMV matrices and the notion of balayage, as described in Sections~\ref{subsec:balay} and~\ref{subsec:distr_asympt_opuc}.


\subsection{Balayage}\label{subsec:balay}

First, recall that for any probability measure $\mu$ with support in $\overline{\D}$, there exists a unique probability measure $\mathcal{P}(\mu)$ supported on $\T$ that satisfies
\begin{equation}
	\int_{\overline{\D}} z^k \dd\mu(z) = \int_{0}^{2\pi} \ee^{ik\theta} \dd\mathcal{P}(\mu)(\ee^{\ii\theta}), \quad k\in\N_{0}.
	\label{eq:balay_moments_eq}
\end{equation}
The passage from $\mu$ to $\mathcal{P}(\mu)$ is called balayage. The balayage measure $\mathcal{P}(\mu)$ can be also defined by the formula
\[
\int_{\T}f(z)\dd\mathcal{P}(\mu)(z)=\int_{\overline{\D}}\mathcal{P}_{*}f(z)\dd\mu(z)
\]
for all continuous functions $f$ on $\T$, where
\[
\mathcal{P}_{*}f(r\ee^{\ii\theta})=\begin{cases}\displaystyle
	\frac{1}{2\pi}\int_{0}^{2\pi}P_{r}(\theta-t)f(\ee^{\ii t})\dd t,& \; \mbox{ if } r<1,\\[12pt]
	f(\ee^{\ii\theta}),& \; \mbox{ if } r=1,
\end{cases}
\]
and
\[
P_{r}(t)=\sum_{n=-\infty}^{\infty}r^{|n|}\ee^{\ii n t}=\frac{1-r^{2}}{1-2r\cos t+r^{2}}
\]
is the Poisson kernel. For more details and the proof, see~\cite[Prop 8.2.2]{simon_opuc1}.

Recall that $\nu_{n,N}$ denotes the zero-counting measure of $\Phi_{n,N}$, see~\eqref{eq:count-meas_var}.

\begin{thm}\label{thm:opuc_balayage}
	Let $t>0$ and $\alpha:[0,t]\to\overline{\D}$ be measurable. Suppose further that $\{\alpha_{n,N} \mid n,N\in\N_{0} \}\subset\D$ is such that
	\[
	\lim_{n/N\to s} \alpha_{n,N}=\alpha(s), \,\mbox{ for {a.e.} } s\in[0,t].
	\]
	Then
	\begin{equation}\label{eq:weak_limit_opuc_balayage}
		\wlim_{n/N\to t} \mathcal{P}(\nu_{n,N})=\frac{1}{t}\int_{0}^{t}\nu_{|\alpha(s)|}\dd s,
	\end{equation}
	where $\nu_{|\alpha(s)|}$ is given by~\eqref{eq:meas_nu} if $|\alpha(s)|\neq1$, and equals $\delta_{-1}$, otherwise.
\end{thm}

\begin{proof}
	Since a negative moment of a positive measure supported on $\T$ is the complex-conjugate of the respective positive moment and Laurent polynomials are dense in $C(\T)$ by the Stone--Weierstrass theorem, it is sufficient to show the convergence of the positive moments of the measures in~\eqref{eq:weak_limit_opuc_balayage}. By~\eqref{eq:balay_moments_eq}, the $k$-th moment of $\mathcal{P}(\nu_{n,N})$ equals
	\[
	\int_{0}^{2\pi} \ee^{\ii k\theta} \dd\mathcal{P}(\nu_{n,N})\big(\ee^{\ii\theta}\big) = \int_{\overline{\D}} z^{k} \dd\nu_{n,N}(z) = \frac{1}{n} \Tr\left(\mathcal{C}_{n,N}\right)^{k}
	= \frac{1}{n} \Tr\left(\mathcal{C}^{(\beta)}_{n,N}\right)^{k}+o(1),
	\]
	{as $n/N\to t$. The last equality holds since matrices} $\mathcal{C}^{(\beta)}_{n,N}$ and $\mathcal{C}_{n,N}$ differ only in the last two rows. {Since 
		\[
		\frac{1}{n} \Tr\left(\mathcal{C}^{(\beta)}_{n,N}\right)^{k}=\int_{\T} z^{k}\dd\nu_{n,N}^{(\beta)}(z),
		\]
		it suffices to apply Theorem~\ref{thm:popuc_asympt_distr} to complete the proof.}
\end{proof}

\subsection{Asymptotic zero distribution of OPUC with variable Verblunsky coefficients}\label{subsec:distr_asympt_opuc}

If the limiting function $\alpha:[0,t]\to\overline{\D}$ does not vanish at the right edge $t$
or if it does but in a certain moderate fashion, Theorem~\ref{thm:popuc_asympt_distr} extends to variable OPUC. To prove this claim we will need the following auxiliary result. 

\begin{lem}\label{lem:balay_meas}
	Let $\Phi_{n}$ be a sequence of OPUC and $\nu_{n}$ be the corresponding sequence of zero-counting measures. Suppose further that, for any $\delta\in(0,1)$, it holds
	\begin{equation}
		\lim_{n\to\infty}{\frac{\sharp\{j\mid|z_{j}|\leq1-\delta\}}{n}}=0,
		\label{eq:lem_zer_assum}
	\end{equation}
	where $\sharp M$ is the cardinality of a set $M$ and $z_{1},\dots,z_{n}$ are the roots of $\Phi_{n}$ (counted with multiplicities). Then
	\[
	\wlim_{n\to\infty}\left(\nu_{n}-\mathcal{P}(\nu_{n})\right)=0.
	\]
\end{lem}

\begin{proof}
	The proof is divided into two parts. For the zero-counting measure $\nu_{n}$, we introduce an auxiliary measure
	\[
	\widetilde{\nu}_{n}=\frac{1}{n}\sum_{j=1}^{n}{\delta}_{\ee^{\ii\theta_{j}}},
	\]
	where we write $z_{j}=r_{j}\ee^{\ii\theta_{j}}$ for some $r_{j}\geq0$ and $\theta_{j}\in(-\pi,\pi]$; if $r_{j}=0$ then $\theta_{j}$ can be taken arbitrarily. In the first part of the proof, we show that
	\begin{equation}
		\wlim_{n\to\infty}(\nu_{n}-\widetilde{\nu}_{n})=0,
		\label{eq:wlim_aux_nu_nu}
	\end{equation}
	while the second part is devoted to the proof of the limit relation
	\begin{equation}
		\wlim_{n\to\infty}(\widetilde{\nu}_{n}-\mathcal{P}(\nu_{n}))=0.
		\label{eq:wlim_aux_nu_balay_nu}
	\end{equation}
	Clearly, \eqref{eq:wlim_aux_nu_nu} and~\eqref{eq:wlim_aux_nu_balay_nu} imply the statement.
	
	1) \emph{The verification of} \eqref{eq:wlim_aux_nu_nu}: Pick $f\in C(\overline{\D})$ and $\epsilon>0$ arbitrarily. Since $\overline{\D}$ is compact, $f$ is uniformly continuous in~$\overline{\D}$, and hence there exists $\delta\in(0,1)$ such that, for any $z,w\in\overline{\D}$, $|w-z|<\delta$, one has
	\begin{equation}
		|f(w)-f(z)|<\epsilon.
		\label{eq:unif_cont_f_inproof}
	\end{equation}
	To proceed further, we write
	\begin{equation}
		\frac{1}{n}\sum_{j=1}^{n}\left|f(z_{j})-f(\ee^{\ii\theta_{j}})\right|=
		\frac{1}{n}\sum_{|z_{j}|\leq1-\delta}\left|f(z_{j})-f(\ee^{\ii\theta_{j}})\right|+
		\frac{1}{n}\sum_{|z_{j}|>1-\delta}\left|f(z_{j})-f(\ee^{\ii\theta_{j}})\right|.
		\label{eq:sum_decomp_inproof}
	\end{equation}
	
	The first term on the right-hand side of~\eqref{eq:sum_decomp_inproof} can be estimated as
	\[
	\frac{1}{n}\sum_{|z_{j}|\leq1-\delta}\left|f(z_{j})-f(\ee^{\ii\theta_{j}})\right|\leq\frac{2\|f\|_{\overline{\D}}}{n}\,\sharp\{j\mid|z_{j}|\leq1-\delta\},
	\]
	where $\|f\|_{\overline{\D}}=\sup_{z\in\overline{\D}}|f(z)|$, which implies that
	\[
	\lim_{n\to\infty}\frac{1}{n}\sum_{|z_{j}|\leq1-\delta}\left|f(z_{j})-f(\ee^{\ii\theta_{j}})\right|=0
	\]
	by assumption~\eqref{eq:lem_zer_assum}. On the other hand, for the second term on the right-hand side of~\eqref{eq:sum_decomp_inproof}, one has
	\[
	\frac{1}{n}\sum_{|z_{j}|>1-\delta}\left|f(z_{j})-f(\ee^{\ii\theta_{j}})\right|\leq
	\frac{\epsilon}{n}\,\sharp\{j\mid|z_{j}|>1-\delta\}\leq\epsilon
	\]
	by~\eqref{eq:unif_cont_f_inproof} because
	\[
	\left|z_{j}-\ee^{\ii\theta_{j}}\right|=1-|z_{j}|<\delta.
	\]
	Taking into account that $\epsilon$ can be chosen arbitrarily small, we conclude that
	\[
	\lim_{n\to\infty}\frac{1}{n}\sum_{j=1}^{n}\left|f(z_{j})-f(\ee^{\ii\theta_{j}})\right|=0,
	\]
	which implies~\eqref{eq:wlim_aux_nu_nu}.
	
	2) \emph{The verification of} \eqref{eq:wlim_aux_nu_balay_nu}: Both measures $\widetilde{\nu}_{n}$ and $\mathcal{P}(\nu_{n})$ are positive measures supported in~$\T$. Therefore, in order to obtain~\eqref{eq:wlim_aux_nu_balay_nu}, it is sufficient to show that
	\begin{equation}
		\lim_{n\to\infty}\int_{0}^{2\pi}\ee^{\ii k t}\dd\!\left(\widetilde{\nu}_{n}-\mathcal{P}(\nu_{n})\right)(\ee^{\ii t})=0
		\label{eq:lim_pos_mom_balay_nu_tilde_nu}
	\end{equation}
	for all $k\in\N_{0}$. For any $k\in\N_{0}$, one can make use of~\eqref{eq:balay_moments_eq} to obtain
	\begin{equation}
		\int_{0}^{2\pi}\ee^{\ii k t}\,\dd\!\left(\widetilde{\nu}_{n}-\mathcal{P}(\nu_{n})\right)(\ee^{\ii t})
		=\int_{0}^{2\pi}\ee^{\ii k t}\dd\widetilde{\nu}_{n}(\ee^{\ii t})-\int_{\overline{\D}}z^{k}\dd\nu_{n}(z)=\frac{1}{n}\sum_{j=1}^{n}\left(\ee^{\ii k\theta_{j}}-z_{j}^{k}\right).
		\label{eq:tilde_nu_balay_nu_inproof}
	\end{equation}
	Choosing $\delta\in(0,1)$ we may write similarly as in the first part of the proof that
	\begin{align}
		\frac{1}{n}\sum_{j=1}^{n}\left|\ee^{\ii k\theta_{j}}-z_{j}^{k}\right|&=
		\frac{1}{n}\sum_{|z_{j}|\leq 1-\delta}\left|\ee^{\ii k\theta_{j}}-z_{j}^{k}\right|+
		\frac{1}{n}\sum_{|z_{j}|> 1-\delta}\left|\ee^{\ii k\theta_{j}}-z_{j}^{k}\right|\nonumber\\
		&\leq\frac{2}{n}\,\sharp\{j\mid|z_{j}|\leq1-\delta\}+\delta.
		\label{eq:mom_diff_estim_inproof}
	\end{align}
	Since $\delta$ can be chosen arbitrarily small, \eqref{eq:tilde_nu_balay_nu_inproof}, \eqref{eq:mom_diff_estim_inproof}, and assumption~\eqref{eq:lem_zer_assum} imply~\eqref{eq:lim_pos_mom_balay_nu_tilde_nu}.
\end{proof}

Now, we are ready to prove our main result on the asymptotic zero distribution of OPUC with varying Verblunsky coefficients.

\begin{thm}\label{thm:opuc_asympt_distr}
	Let the assumptions of Theorem~\ref{thm:popuc_asympt_distr} hold. Suppose, in addition, that 
	\begin{equation}
		\lim_{n/N\to t}|\alpha_{n-1,N}|^{1/n}=1.
		\label{eq:asum_opuc_var}
	\end{equation}
	Then
	\[
		\wlim_{n/N\to t} \nu_{n,N}=\frac{1}{t}\int_{0}^{t}\nu_{|\alpha(s)|}\dd s.
	\]
\end{thm}
\begin{rem}
	Note that~\eqref{eq:asum_opuc_var} is automatically fulfilled if {$\alpha_{n,N}\to\alpha(t)\neq0$ as $n/N\to t$.}
\end{rem}
\begin{rem}\label{rem:asymptPerOPUC}
	The results of Theorems~\ref{thm:opuc_balayage} and~\ref{thm:opuc_asympt_distr} can be extended to the asymptotically periodic setting of Theorem~\ref{thm:asympt_distr_popuc_periodic} using the same arguments.
\end{rem}
\begin{proof}[Proof of Theorem~\ref{thm:opuc_asympt_distr}]
	The idea of the proof relies on showing that the property~\eqref{eq:lem_zer_assum} holds in the variable setting, i.e the roots $z_{1,n,N},\dots,z_{n,n,N}$ of $\Phi_{n,N}$ are approaching the unit circle up to $o(n)$ exceptions as $n/N\to t$.
	
	Choose subsequences $\{n_{j}\},\{N_{j}\}\subset\N$ such that $n_{j},N_{j}\to\infty$ and $n_{j}/N_{j} \to t$. Recall the identity
	\[
	-\bar{\alpha}_{n_{j}-1,N_{j}}=\Phi_{n_{j},N_{j}}(0)=\prod_{k=1}^{n_j}\left(-z_{k,n_{j},N_{j}}\right).
	\]
	Having~\eqref{eq:asum_opuc_var}, we can assume that $\alpha_{n_{j}-1,N_{j}}\neq0$ for all $j$ large.
	Then we can take logarithms in the last equation {and} obtain 
	\begin{equation}
		\frac{1}{n_j} \log|\alpha_{n_j-1,N_j} | = \frac{1}{n_j} \sum_{k=1}^{n_j} \log|z_{k,n_j,N_j}| \leq \frac{\log(1-\delta)}{n_j}\,\sharp\{k \mid |z_{k,n_j,N_j}|\leq1-\delta\}
		\label{eq:log_det_at_zero}
	\end{equation}
	for arbitrary $\delta\in(0,1)$. By the assumption~\eqref{eq:asum_opuc_var}, the left-hand side of~\eqref{eq:log_det_at_zero} goes to zero for $j\to\infty$. Taking into account that the right-hand side of~\eqref{eq:log_det_at_zero} is non-positive, one readily obtains
	\[
	\lim_{j\to\infty}{\frac{\sharp\{k \mid |z_{k,n_j,N_j}|\leq1-\delta\}}{n_j}}=0.
	\]
	Therefore we can apply Lemma~\ref{lem:balay_meas} which together with Theorem~\ref{thm:opuc_balayage} finishes the proof.\
\end{proof}

\begin{rem}
	The argument used in this proof is directly borrowed from the proof of~\cite[Thm 8.1.1]{simon_opuc1}.
\end{rem}

\begin{rem}\label{rem:open_prob}
	It would be very interesting to establish the variable version of the Mhaskar--Saff theorem, that is, to determine the limiting zero density of OPUC when the limit in~\eqref{eq:asum_opuc_var} is no longer assumed to be $1$. Addressing this problem likely requires a different approach. We leave this question as an open problem for future. 
\end{rem}

\section{Examples}\label{sec:examples}

As an illustration, we provide examples of POPUC and OPUC with variable Verblunsky coefficients whose asymptotic zero distribution is determined by Theorems \ref{thm:popuc_asympt_distr} and~\ref{thm:opuc_asympt_distr} and can be computed in an even more explicit form. This concerns particularly cases when $|\alpha|$ is a monotone function on $(0,t)$.

\begin{prop}\label{prop:dens_sigma_t_f_monot}
	Let $\alpha:[0,t]\to\overline{\D}$ be continuous and denote $f(s)=|\alpha(s)|$ for $s\in[0,t]$.
	\begin{enumerate}[label=(\roman*)]
		\item If $f$ is strictly increasing in $(0,t)$, then $\sigma_{t}$ is the absolutely continuous measure with $\supp\sigma_{t}=\{\ee^{\ii\theta} \mid \theta\in[\theta_{f(0)},2\pi-\theta_{f(0)}]\}$ and
		\begin{equation}
			\frac{\dd\sigma_{t}}{\dd\theta}\left(\ee^{\ii\theta}\right)=\frac{1}{2\pi t}\int_{0}^{\min(t,f^{-1}(\sin(\theta/2))}\frac{\sin\left(\theta/2\right)}{\sqrt{\sin^{2}\left(\theta/2\right)-f^{{2}}(s)}}\dd s,
			\label{eq:dens_sigma_t_f_incr}
		\end{equation}
		for $\theta\in(\theta_{f(0)},2\pi-\theta_{f(0)})$.
		\item If $f$ is strictly decreasing in $(0,t)$, then $\sigma_{t}$ is the absolutely continuous measure with $\supp\sigma_{t}=\{\ee^{\ii\theta} \mid [\theta_{f(t)},2\pi-\theta_{f(t)}]\}$ and
		\begin{equation}
			\frac{\dd\sigma_{t}}{\dd\theta}\left(\ee^{\ii\theta}\right)=\frac{1}{2\pi t}\int_{\max(0,f^{-1}(\sin(\theta/2))}^{t}\frac{\sin\left(\theta/2\right)}{\sqrt{\sin^{2}\left(\theta/2\right)-f^{{2}}(s)}}\dd s,
			\label{eq:dens_sigma_t_f_decr}
		\end{equation}
		for $\theta\in(\theta_{f(t)},2\pi-\theta_{f(t)})$.
	\end{enumerate}
	Here $\sigma_{t}$ is the average measure~\eqref{eq:def_sigma_t} and the angle $\theta_{a}$ is defined by~\eqref{eq:def_theta}.
\end{prop}

\begin{proof}
	We verify the statement for $f$ strictly increasing. If $f$ is strictly decreasing, the proof is analogous.

	Recall that $\supp\nu_{f(s)}=\{\ee^{\ii\theta} \mid \theta\in[\theta_{f(s)},2\pi-\theta_{f(s)}]\}$ 
	and therefore, assuming $f$ is increasing, it follows readily from~\eqref{eq:def_sigma_t} that $\supp\sigma_{t}=\{\ee^{\ii\theta} \mid \theta\in[\theta_{f(0)},2\pi-\theta_{f(0)}]\}$.
	
	Next, for $\xi\in(\theta_{f(0)},2\pi-\theta_{f(0)})$, we investigate the distribution function
	\[
	F_{\sigma_{t}}(\xi)=\sigma_{t}\left(\left\{\ee^{\ii\theta} \mid \theta\in\left[\theta_{f(0)},\xi\right]\right\}\right)=\frac{1}{t}\int_{0}^{t}\nu_{f(s)}\left(\left\{\ee^{\ii\theta} \mid \theta\in\left[\theta_{f(0)},\xi\right]\right\}\right)\dd s.
	\]
	{Since $\sigma_{t}$ is a symmetric measure with respect to the real line we can additionally assume that $\xi\leq\pi$. Then,} with the aid of the definition~\eqref{eq:meas_nu}, one computes that
	\[
	F_{\sigma_{t}}(\xi)=\frac{1}{2\pi t}\int_{0}^{\min(t,f^{-1}(\sin(\xi/2))}\left(\int_{\theta_{f(s)}}^{\xi}\frac{\sin\left(\theta/2\right)}{\sqrt{\cos^{2}\left(\theta_{f(s)}/2\right)-\cos^{2}\left(\theta/2\right)}}\dd\theta\right)\dd s.
	\]
	By interchanging the order of integrals and using~\eqref{eq:def_theta}, one arrives at the expression
	\[
	F_{\sigma_{t}}(\xi)=\frac{1}{2\pi t}\int_{\theta_{f(0)}}^{\xi}\left(\int_{0}^{\min(t,f^{-1}(\sin(\theta/2))}\frac{\sin\left(\theta/2\right)}{\sqrt{\sin^{2}\left(\theta/2\right)-f^{2}(s)}}\dd s\right)\dd\theta,
	\]
	from which~\eqref{eq:dens_sigma_t_f_incr} follows {for $\theta\in(\theta_{f(0)},\pi)$. Since the resulting expression for density~\eqref{eq:dens_sigma_t_f_incr} is invariant under the change of variable $\theta$ by $2\pi-\theta$, it holds true for all $\theta\in(\theta_{f(0)},2\pi-\theta_{f(0)})$.}
\end{proof}

Next examples illustrate the asymptotic zero distributions of variable OPUC given by Theorem~\ref{thm:opuc_asympt_distr} and formulas~\eqref{eq:dens_sigma_t_f_incr} and~\eqref{eq:dens_sigma_t_f_decr} for several specific choices of $f$.

\begin{example}
	Suppose $f(s)=s^{\omega}$ and $0<t\leq 1$, where $\omega>0$ is a parameter. This situation corresponds to the Verblunsky coefficients $\alpha_{n,N}=\left(n/N\right)^{\omega}$ for $0\le n\le N$, for example. Clearly, $f$ is strictly increasing and, by Proposition~\ref{prop:dens_sigma_t_f_monot}, $\supp\sigma_{t}=\T$.
	
	First, suppose $\theta\in(0,2\pi)$ and $\sin(\theta/2)>t^{\omega}$, then, according to~\eqref{eq:dens_sigma_t_f_incr}, we have
 \begin{align*}
		\frac{\dd\sigma_{t}}{\dd\theta}\left(\ee^{\ii\theta}\right)&=\frac{1}{2\pi t}\int_{0}^{t}\frac{\sin\left(\theta/2\right)}{\sqrt{\sin^{2}\left(\theta/2\right)-s^{2\omega}}}\dd s\\
		&=\frac{1}{2\pi t}(\sin\left(\theta/2\right))^{1/\omega}\int_{0}^{t/(\sin(\theta/2))^{1/\omega}}\frac{\dd u}{\sqrt{1-u^{2\omega}}},
 \end{align*}
where we substituted for $s=(\sin(\theta/2))^{1/\omega} u$. The above integral can be expressed in terms of standard special functions. Namely, using the definition of the Gauss hypergeometric series, see~\cite[Chap.~15]{dlmf}, and expanding the integrand into power series of variable $u^{2\omega}$, one verifies the identity
	\[
	\int_{0}^{x}\frac{\dd u}{\sqrt{1-u^{2\omega}}}=x\,_{2}F_{1}\left(\frac{1}{2},\frac{1}{2\omega};\frac{2\omega+1}{2\omega};x^{2\omega}\right), \quad x\in(0,1).
	\]
	Hence, the density of $\sigma_{t}$ can be written as
	\[
	\frac{\dd\sigma_{t}}{\dd\theta}\left(\ee^{\ii\theta}\right)=\frac{1}{2\pi}\,_{2}F_{1}\left(\frac{1}{2},\frac{1}{2\omega};\frac{2\omega+1}{2\omega};\frac{t^{2\omega}}{\sin^{2}(\theta/2)}\right).
	\]
	
	On the other hand, if $\theta\in(0,2\pi)$ and $\sin(\theta/2)\leq t^{\omega}$, one proceeds similarly using~\eqref{eq:dens_sigma_t_f_incr} to obtain
 \[
\frac{\dd\sigma_{t}}{\dd\theta}\left(\ee^{\ii\theta}\right)=\frac{\left(\sin\left(\theta/2\right)\right)^{\frac{1}{\omega}}}{2\pi t} \,_{2}F_{1}\left(\frac{1}{2},\frac{1}{2\omega};\frac{2\omega+1}{2\omega};1\right)
=\frac{\left(\sin\left(\theta/2\right)\right)^{\frac{1}{\omega}}}{2\sqrt{\pi}t}\,\frac{\Gamma\left(\frac{2\omega+1}{2\omega}\right)}{\Gamma\left(\frac{\omega+1}{2\omega}\right)},
 \]
where the last equality holds due to the identity~\cite[Eq.~(15.4.20)]{dlmf}.
	
In summary, for $\theta\in(0,2\pi)$, the density reads
	\begin{equation}
		\frac{\dd\sigma_{t}}{\dd\theta}\left(\ee^{\ii\theta}\right)=\begin{cases}
			\frac{1}{2\pi}\,_{2}F_{1}\left(\frac{1}{2},\frac{1}{2\omega};\frac{2\omega+1}{2\omega};\frac{t^{2\omega}}{\sin^{2}(\theta/2)}\right), & \quad\mbox{ if } \sin(\theta/2)>t^{\omega},\\[8pt]
			\frac{1}{2\sqrt{\pi} t}\frac{\Gamma\left((2\omega+1)/(2\omega)\right)}{\Gamma\left((\omega+1)/(2\omega)\right)}\left(\sin\left(\theta/2\right)\right)^{\frac{1}{\omega}}, & \quad\mbox{ if } \sin(\theta/2)\leq t^{\omega}.
		\end{cases}
		\label{eq:dens_f_power_general}
	\end{equation}
	In particular, for $\omega=1$, the expression for the density further simplifies
	\begin{equation}
		\frac{\dd\sigma_{t}}{\dd\theta}\left(\ee^{\ii\theta}\right)=\begin{cases}
			\frac{1}{2\pi t}\sin\left(\frac{\theta}{2}\right)\arcsin\left(\frac{t}{\sin(\theta/2)}\right), & \quad\mbox{ if } \sin(\theta/2)>t,\\[6pt]
			\frac{1}{4t}\sin\left(\frac{\theta}{2}\right), & \quad\mbox{ if } \sin(\theta/2)\leq t.
		\end{cases}
		\label{eq:dens_f_power_om_1}
	\end{equation}
	These densities are plotted for $\omega\in\{1,2\}$ and several choices of $t\in(0,1]$ in Figure~\ref{fig:dens1ab}.
\end{example}

\begin{figure}[htb!]
	\centering
	\begin{subfigure}[b]{0.49\textwidth}
		\includegraphics[width=\textwidth]{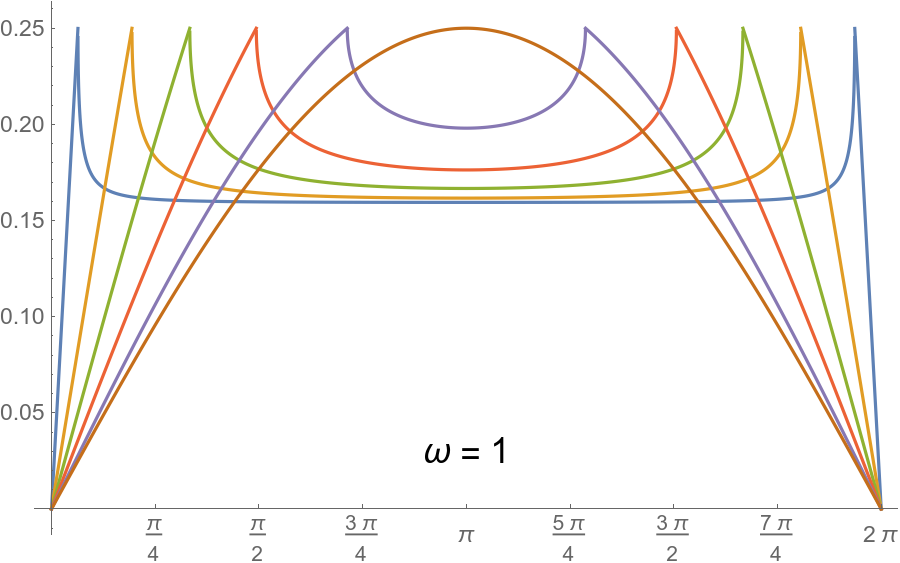}
	\end{subfigure}
	\begin{subfigure}[b]{0.49\textwidth}
		\includegraphics[width=\textwidth]{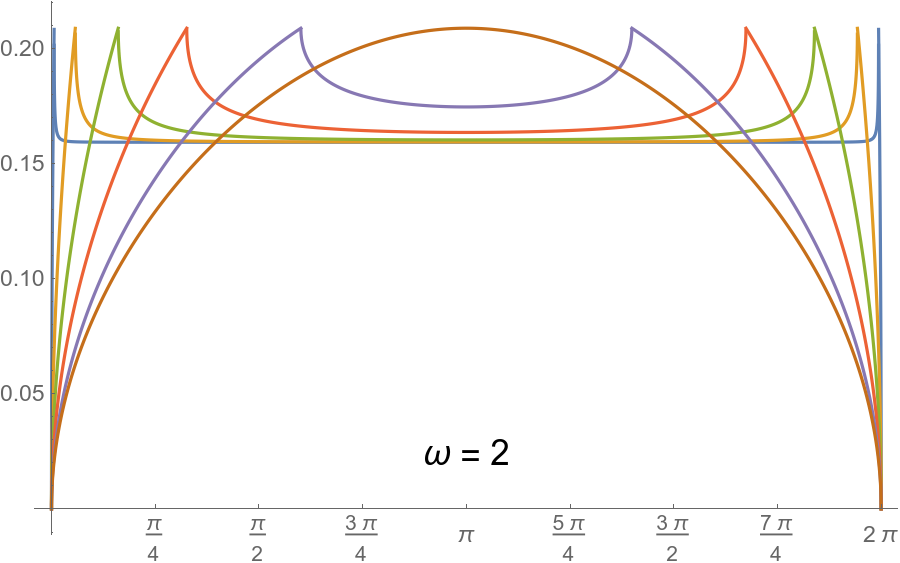}
	\end{subfigure}
	\captionsetup{width=0.99\textwidth}
	\caption{Densities~\eqref{eq:dens_f_power_general} and~\eqref{eq:dens_f_power_om_1} for $\omega\in\{1,2\}$ and $t\in\{0.1,0.3,0.5,0.7,0.9,1\}$.}
	\label{fig:dens1ab}
\end{figure}

\begin{example}
	The Rogers--Szeg\H{o} polynomials~\cite[Ex.~1.6.5]{simon_opuc1} correspond to the choice of Verblunsky parameters $\alpha_n = (-1)^n \gamma^{n+1}$ with $\gamma\in(0,1)$; see also~\cite[Chap.~17]{ismail_09}.  Their orthogonality measure is supported on $\T$ with the density
	\begin{equation}\label{eq:RogSzeMeas}
		\frac{\dd\mu}{\dd\theta}(\ee^{\ii \theta}) = \frac{1}{\sqrt{2\pi a}} \sum_{j=-\infty}^{\infty} \ee^{-(\theta-2\pi j)^2/(2a)},
	\end{equation}
	where $a=-2\log\gamma$. 
	The $n$-th OPUC can {be expressed explicitly as}
	$$
	\Phi_n(z)
	= \sum_{j=0}^{n}{\frac{(-1)^{j}\gamma^{j}}{(1-\gamma^2)\ldots(1-\gamma^{2j})}\, z^{n-j}}.
	$$
	All of its zeros belong to the circle $\{z\in\C \mid |z| = \gamma\}$, {see~\cite{maz-etal_90}}. By the Mhaskar--Saff theorem, the zero counting measure weakly converges to the uniform measure on this circle.
	
	Let us now study the varying Rogers--Szeg\H{o} polynomials where $\gamma$ is $N$-dependent. The simplest setting to start with is to take $\gamma_{N}=\zeta^{1/N}$ with a fixed $\zeta\in (0,1)$.  The orthogonality measure $\mu_N$ is then~\eqref{eq:RogSzeMeas} with $a$ replaced by $a_N=-2\log(\zeta)/N$. From the non-varying case it is clear that the zeros must converge to $\T$, however the limit of the zero counting measure is no longer uniform, as we show now.
	
	{Instead of analyzing} the case $\alpha_{n,N}=(-1)^n \zeta^{(n+1)/N}$, {which corresponds to the choice $\gamma_{n}=\zeta^{1/N}$, we consider the rotated} case $\alpha_{n,N}=\zeta^{(n+1)/N}$. Let us therefore study the latter case. As $n/N\to t$, we obtain the setting of Theorem~\ref{thm:opuc_asympt_distr} and Proposition~\ref{prop:dens_sigma_t_f_monot} with  strictly decreasing  function $f(s) = \zeta^{s}$.
		
It follows that the limiting measure $\sigma_{t}$ of OPUC as well as POPUC is absolutely continuous with 
	\[
	\supp\sigma_{t}=\{\ee^{\ii\theta} \mid \theta\in[2\arcsin \zeta^{t},2\pi-2\arcsin \zeta^{t}]\}
	\]
	and density given by~\eqref{eq:dens_sigma_t_f_decr}. Note that $f^{-1}(\sin(\theta/2)) = \log_\zeta(\sin(\theta/2))>0$; so for any $\theta\in(2\arcsin \zeta^{t},2\pi-2\arcsin \zeta^{t})$ we have
	\[
	\frac{\dd\sigma_{t}}{\dd\theta}\!\left(\ee^{\ii\theta}\right)=\frac{1}{2\pi t}\int_{\log_{\zeta}\sin(\theta/2)}^{t}\frac{\sin\left(\theta/2\right)}{\sqrt{\sin^{2}\left(\theta/2\right)-\zeta^{2s}}}\dd s=-\frac{1}{2\pi t\log \zeta}\int_{\frac{\zeta^{t}}{\sin(\theta/2)}}^{1}\frac{\dd u}{u\sqrt{1-u^{2}}},
	\]
	where we substituted for $\zeta^{s}=u\sin(\theta/2)$. Since
	\[
	\frac{\dd}{\dd u}\left(\log u - \log(1+\sqrt{1-u^{2}})\right)=\frac{1}{u\sqrt{1-u^{2}}},
	\]
	the final formula for the density reads
	\begin{equation}
		\frac{\dd\sigma_{t}}{\dd\theta}\left(\ee^{\ii\theta}\right)=\frac{1}{2\pi}\left[1-\frac{1}{t}\log_{\zeta}\!\left(\sin(\theta/2)+\sqrt{\sin^{2}(\theta/2)-\zeta^{2t}}\right)\right],
		\label{eq:dens_f_exp}
	\end{equation}
	for $\theta\in(2\arcsin \zeta^{t},2\pi-2\arcsin \zeta^{t})$. Density~\eqref{eq:dens_f_exp} is plotted in Figure~\ref{fig:dens2} for several choices of $\zeta\in(0,1)$.
	
	The limiting zero density for the varying Rogers--Szeg\H{o} polynomials with $\alpha_{n,N}=(-1)^n \zeta^{(n+1)/N}$ is then $\sigma_{t}(\ee^{\ii(\theta-\pi)})$, where $\sigma_t$ is ~\eqref{eq:dens_f_exp}. 
\end{example}

\begin{figure}[htb!]
	\centering
	\includegraphics[width=0.9\textwidth]{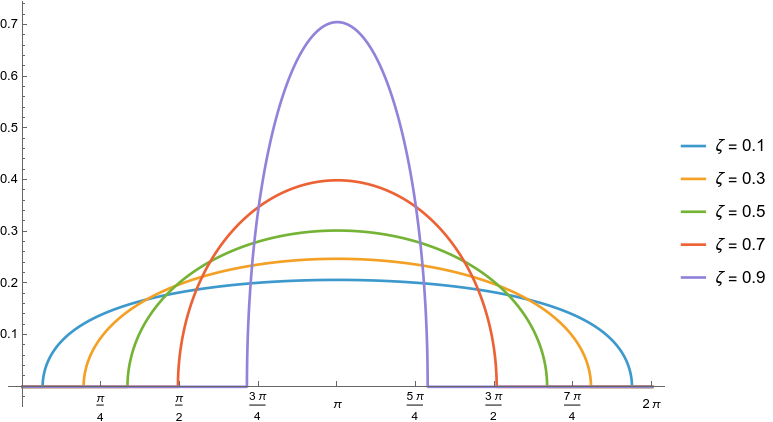}
	\captionsetup{width=0.99\textwidth}
	\caption{Density~\eqref{eq:dens_f_exp} plotted for $t=1$ and $\zeta\in\{0.1,0.3,0.5,0.7,0.9\}$.}
	\label{fig:dens2}
\end{figure}

In the next example we study a case when the density of~$\sigma_{t}$ has a singularity.

\begin{example}
	Let $f(s)=\sqrt{1-s^{2}}$ and $0<t\leq1$. Then $f$ is strictly decreasing and a straightforward computation of the integral in~\eqref{eq:dens_sigma_t_f_decr} results in the formula
	\begin{equation}
		\frac{\dd\sigma_{t}}{\dd\theta}\left(\ee^{\ii\theta}\right)=\frac{\sin\left(\theta/2\right)}{2\pi t}\left[\log\left(t+\sqrt{t^{2}-\cos^{2}\left(\theta/2\right)}\right)-\log\left|\cos\left(\theta/2\right)\right|\right],
		\label{eq:dens_f_sing}
	\end{equation}
	for $\theta\in[2\arcsin\sqrt{1-t^{2}},\pi)\cup(\pi,2\pi-2\arcsin\sqrt{1-t^{2}}]$. Note the singularity at $\theta=\pi$ of the function~\eqref{eq:dens_f_sing}; see Figure~\ref{fig:dens3} for numerical plots.
\end{example}

\begin{figure}[htb!]
	\centering
	\includegraphics[width=0.9\textwidth]{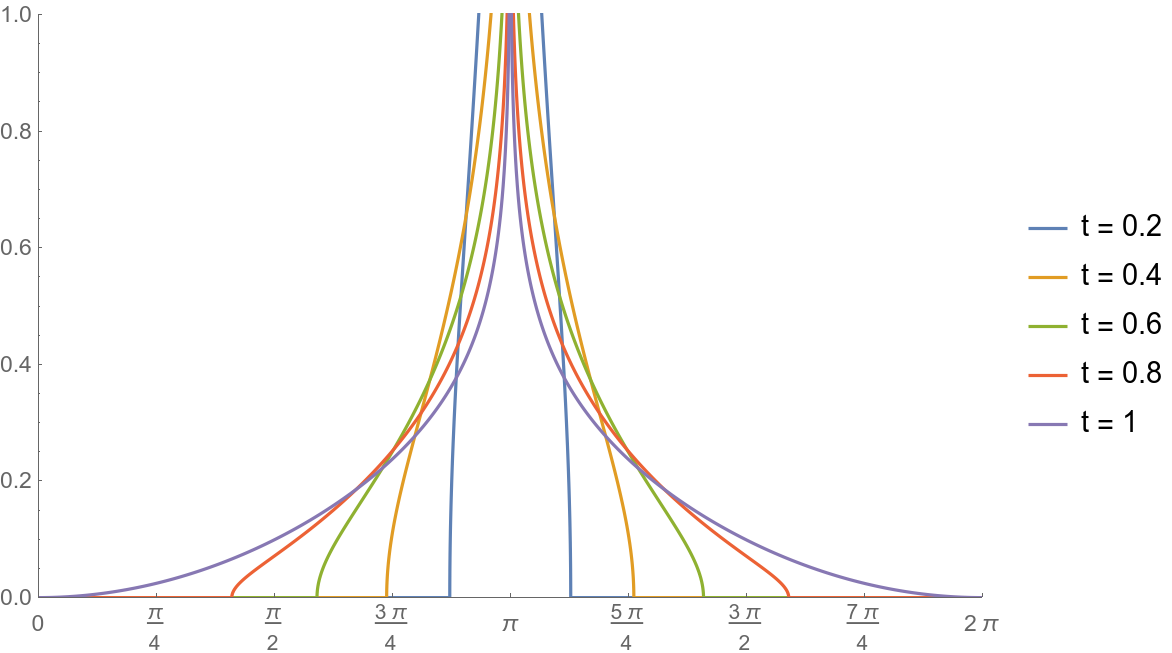}
	\captionsetup{width=0.99\textwidth}
	\caption{Density~\eqref{eq:dens_f_sing} plotted for $t\in\{0.2,0.4,0.6,0.8,1\}$.}
	\label{fig:dens3}
\end{figure}

The next example is interesting since it illustrates a situation where the limiting distribution of roots of variable OPUC is controlled by the complete elliptic integral of the first kind
\begin{equation}
	K(k)=\int_{0}^{1}\frac{\dd x}{\sqrt{(1-x^{2})(1-k^{2}x^{2})}}, \quad k\in(0,1),
	\label{eq:def_elliptic_int_K}
\end{equation}
see~\cite[Chap.~19]{dlmf}.
Moreover, it indicates that the applicability of Proposition~\ref{prop:dens_sigma_t_f_monot} can be extended to certain not necessarily monotonic functions~$f$.

\begin{example}
	Let $f(s)=\sin\left(\pi s/t\right)$ for $t>0$. Then $f$ is not monotonic on $(0,t)$ but is monotonic separately on $(0,t/2)$ and $(t/2,t)$. Moreover, $f$ is symmetric with respect to~$t/2$, i.e., $f(s)=f(t-s)$, for all $s\in[0,t]$. Taking these properties into account, it is easy to see from~\eqref{eq:def_sigma_t} that
	\[
	\sigma_{t}=\frac{2}{t}\int_{0}^{t/2}\nu_{f(s)}\dd s.
	\]
	Hence the problem can be reduced to the case when $f$ is replaced by
	\[
	g(s)=\sin\left(\frac{\pi s}{2t}\right), \quad s\in[0,t].
	\]
	
	Since $g$ is strictly increasing in $(0,t)$ we may apply Proposition~\ref{prop:dens_sigma_t_f_monot}. Denoting again by $\sigma_{t}$ the average measure corresponding to $|\alpha|=g$ and assuming $\theta\in(0,\pi)$, formula~\eqref{eq:dens_sigma_t_f_incr} yields
	\begin{align*}
		\frac{\dd\sigma_{t}}{\dd\theta}(\ee^{\ii\theta})&=\frac{1}{2\pi t}\int_{0}^{t\theta/\pi}\frac{\sin\left(\theta/2\right)}{\sqrt{\sin^{2}\left(\theta/2\right)-\sin^{2}\left(\pi s/2t\right)}}\dd s \\
		&=\frac{1}{\pi^{2}}\sin\left(\frac{\theta}{2}\right)\int_{0}^{1}\frac{\dd u}{\sqrt{\left(1-u^{2}\right)\left(1-\sin^{2}\left(\theta/2\right)u^{2}\right)}},
	\end{align*}
	where we substituted for $\sin\left(\pi s/2t\right)=u\sin\left(\theta/2\right)$. Using definition~\eqref{eq:def_elliptic_int_K}, we obtain
	\begin{equation}
		\frac{\dd\sigma_{t}}{\dd\theta}(\ee^{\ii\theta})=\frac{\sin\left(\theta/2\right)}{\pi^{2}}\,K\!\left(\sin\left(\frac{\theta}{2}\right)\right),
		\label{eq:dens_f_sine}
	\end{equation}
	for $\theta\in(0,\pi)$. Since $\sigma_{t}$ is symmetric with respect to the real line formula~\eqref{eq:dens_f_sine} extends in the same form to $\theta\in(\pi,2\pi)$. Density~\eqref{eq:dens_f_sine} has a singularity at the middle point $\theta=\pi$ since
	\[
	K(k)=-\log\sqrt{1-k^{2}}+O(1), \quad\mbox{ as } k\to1-,
	\]
	see~\cite[Eq.~(19.12.1)]{dlmf}.
\end{example}

\section*{Acknowledgement}

The research of F.~{\v S}. was supported by the GA{\v C}R grant No. 20-17749X.

\bibliographystyle{acm}
\bibliography{references_varopuc}

\end{document}